\documentclass[11pt, reqno]{amsart}
\usepackage{paper_style}
\usepackage[csfont, paper]{std_math}
\usepackage{local}
	
\IfFileExists{Sections/abstract.tex}{}{}

\hypersetup{pdfauthor={Konstantin Andritsch}
	,pdftitle="Density of almost squares of horospherical orbits in non-uniform quotients of SL2xSL2"
    ,urlcolor=blue
	,citecolor=red
	,linkcolor=blue
	,colorlinks=true
	,colorlinks=true
}

\begin{document}

	%auto-ignore
\title{Density of almost squares of horospherical orbits in non-uniform quotients of $\SL_2(\RR)\times\SL_2(\RR)$}

\author{Konstantin Andritsch}
\address{Department of Mathematics, ETH Zürich, Zürich, Switzerland}
\email{konstantin.andritsch@math.ethz.ch}
%\email[Noy Soffer Aranov]{noysofferaranov@math.utah.edu}
%\email[Zuo Lin]{zul003@ucsd.edu}
%\email[Yuping Ruan]{ruanyp@northwestern.edu}
%\email[Jiyao Tang]{jiyaotang2027@u.northwestern.edu}

\thanks{The author gratefully acknowledges support from the Swiss National Science Foundation (grant 10003145).}

%\keywords{equidistribution}
%\date{\today.}
	
	%auto-ignore
\begin{abstract}	
    Let $\Gamma\subset\SL_2(\RR)\times\SL_2(\RR)$ be an irreducible, non-uniform lattice and define the space $X = \SL_2(\RR)\times\SL_2(\RR)/\Gamma$. Let $U$ be the standard horospherical subgroup in $\SL_2(\RR)\times\SL_2(\RR)$. We show that for every $\delta>0$ and $x\in X$ with dense $U$-orbit, the $U$-orbit evaluated at almost squares, $\{(u_{(n^{2-\delta},m^{2-\delta})}~:~n,m\in\NN\}\cdot x$, is dense in $X$.  
    The main idea of the proof is to shadow large periodic $U$-orbits and reduce the statement to a density statement within the periodic $U$-orbit, \ie a statement for the $2$-torus $\TT^2$. This then follows from an effective version of Weyl's inequality to deduce the result. The effective shadowing of periodic $U$-orbits is achieved using tools from homogeneous dynamics, namely quantitative non-divergence of horospherical subgroups, recurrence under the diagonal flow and effective equidistribution of expanding horospherical subgroups.    
    
    Our approach follows the strategy of the work \cite{KR25} of \citeauthor{KR25} who established density of almost squares of the horocycle flow for non-uniform lattices in $\SL_2(\RR)$.
\end{abstract}
    
	\maketitle
    
	%%%%%%%%%%%%%%%%%%%%%%%%%%%%%%%%%%%%%%%%%%%%%%%%%%%%%%%%%%%

	\section{Introduction}

The study of unipotent subgroups is a classical topic in homogeneous dynamics. 
By the seminal papers of \citeauthor{Ra91a} \cite{Ra91a,Ra91b} orbit closure's and invariant probability measures of unipotent subgroups are by now well-understood.
More than that, even effective equidistribution statements for horospherical orbits \cite{KM96} exist.
Due to their often algebraic nature it is natural to study horospherical orbits at a discrete set of times. For example, in \cite{Ma10, EMSS16} the authors study and prove equidistribution of primitive rational points on expanding horospheres in $\SL_n(\RR)$, and \cite{Lu21} extends these results to products of $\SL_2(\RR)$. 
While equidistribution of the full set of integer times $\NN$ of the horocycle flow in quotients of $\SL_2(\RR)$ follows from the corresponding continuous statements, much less is known once one shifts to a sparser subset of $\NN$. Margulis conjectured that the horocycle flow at the subset of primes should be equidistributed, while Shah conjectured the same should be true for the subset $\{n^{1+\eta}\st n\in\NN\}$ for some $\eta >0$. \citeauthor{Ve10} \cite{Ve10} was able to make significant progress on the conjecture of Shah by proving equidistribution of the horocylce flow when evaluated at times $\{n^{1+\eta}\st n\in\NN\}$ for some small $\eta >0$, for \highlight{cocompact} lattices. The range of $\eta$ was later improved by \citeauthor{FF16} \cite{FF16} to achieve the range $0 < \eta < 1/12$. \citeauthor{SU15} established certain non-concentration properties of the horocycle flow at prime times \cite{SU15}, and \citeauthor{McA19} generalized their approach to obtain asymptotic distribution results of horospherical orbits evaluated at almost-primes (integers with a bound on the number of prime factors). 

More recently, \citeauthor{FKR24} \cite{FKR24} made significant progress towards both conjectures by proving equidistribution of the horocycle flow when evaluated at $2$-primes, that is integers with at most $2$ prime factors. Their methods rely on a celebrated breakthrough of \citeauthor{LMW25} \cite{LMW25} providing quantitative ergodicity for the square of the horocycle flow.
Finally, \citeauthor{KR25} \cite{KR25} established density of almost-squares of the horocycle flow  and density at prime times conditional under the Hardy-Littlewood conjecture. 

This paper extends the results in \cite{KR25} to the study of horospherical orbits evaluated at almost-squares in non-uniform quotients of $\SL_2(\RR)\times\SL_2(\RR)$. More precisely, let $G=\SL_2(\RR)\times\SL_2(\RR)$, assume that $\Gamma\leq G$ is a non-uniform, irreducible lattice and set $X=\tquot{G}{\Gamma}$. Recall that $\Gamma\leq G$ is irreducible, if the projection of $\Gamma$ to each factor is dense. Let
\[ U = \left\{\left(\smat{1}{r_1}{}{1},\smat{1}{r_2}{}{1}\right)\st r_1,r_2\in\RR\right\} \]
be the standard horospherical subgroup in $G$. 
In this paper we prove the following \namecref{thm:main}.
\begin{theorem}\label{thm:main}
    Suppose $x_0 \in X$ is an initial point so that the $U$-orbit $Ux_0$ is dense in $X$. Then for all $1\geq \delta > 0$, the set
    \begin{align*}
        \{(u_{m^{2 - \delta}} \times u_{n^{2 - \delta}}) \st m,n \in \NN\}\cdot x_0
    \end{align*}
    is dense in $X$. 
\end{theorem}

\begin{remark}
    The condition about density of the $U$-orbit $Ux_0$ of the initial point $x_0\in X$ is only a weak constraint. Indeed, by Ratner's orbit closure theorem \cite{Ra91b} for $\SL_2(\RR)\times\SL_2(\RR)$, there is a periodic-dense dichotomy of $U$-orbits, \ie every $U$-orbit is either periodic or dense. The reason is that there are no algebraic subgroups $U\lneq L\lneq G$, such that $L$-orbits in $X$ support finite $L$-invariant measures.
    Clearly, if $x_0$ has a periodic $U$-orbit then $\{(u_{m^{2 - \delta}} \times u_{n^{2 - \delta}}) \st m,n \in \NN\}\cdot x_0$ cannot be (arbitrarily) dense. Thus, the condition about density of $Ux_0$ is necessary but only a soft constraint.

    The same statement is true in case that $\Gamma$ is a reducible lattice of the form $\Gamma_1\times\Gamma_2$, where $\Gamma_1$ and $\Gamma_2$ are non-uniform. However, then the result is a direct consequence of \cite{KR25} applied to each factor individually. Thus, we chose to only treat the case of irreducible lattices.
\end{remark}

\subsection{Outline of the proof}
The strategy of the proof of \Cref{thm:main} is as follows. In rough terms, the density of the set $\{(u_{m^{2 - \delta}} \times u_{n^{2 - \delta}}) \st n \in \NN\}\cdot x_0$ is reduced to a density statement of $\{(u_{m^{2 - \delta}} \times u_{n^{2 - \delta}}) \st m,n \in \NN\}\cdot \omega$ where $\omega\in X$ is a point with a large, periodic $U$-orbit. 

To achieve the reduction to the density statement for periodic $U$-orbits, we need to have a quantitative control of the distance between the $U$-orbit $Ux_0$ and the periodic $U$-orbit $U\omega$. More precisely, given $\delta > 0$, we show (\Cref{prop:approaching-omega}) that a given point $\omega$ with periodic $U$-orbit can be approximated at distance of the order $R^{-1/2}$ at scale $\br\in[\frac{D}{4}R,DR]^2$, for some absolute constant $D$ and some large $R\in\RR$. This is achieved using three key ingredients from homogeneous dynamics. These are quantitative non-divergence of horospherical subgroups (\Cref{lem:quan-non-div}), $A$-recurrence for $A$ the positive diagonal subgroup in $G$ (\Cref{lem:A-recurrence}) and effective equidistribution of expanding horospherical subgroups (\Cref{lem:exp-horospheres}).

Once a point $\omega$ with periodic $U$-orbit is approximated at the correct scale, it is necessary to ensure that the relevant $U$-orbits stay close (\Cref{cor:staying-close-to-closed-orbits}) to each other for a range of size $[0,R^{1/2-\delta/5}]$. Notably, due to the offset between $u_\br$ and $\omega$ in the opposite horospherical direction, there is a shear $\varphi(\br)$ between the parameters of the $U$-orbits of $u_\br x_0$ and $\omega$. Hence, in fact, one needs to prove a density result for the set $\{(u_{\varphi(m^{2 - \delta},n^{2 - \delta})}) \st m,n \in \NN\}\cdot \omega$. 

In order to address this shear, we first approximate $\varphi$ with a Taylor polynomial of large enough degree \Cref{lem:taylor}. We then linearize the argument in the parameters $m$ and $n$ (\Cref{lem:linearize-argument}) and reduce the required density result to a density result on the ordinary $2$-torus $\TT^2 = \tquot{\RR^2}{\ZZ^2}$. This is then a consequence of an effective version of Weyl's inequality (\Cref{lem:weyls-inequality}).

Since large periodic horospherical orbits are known to equidistribute in $X$, one can choose a large enough periodic $U$-orbit $U\omega$ in the beginning, to deduce $\eps$-density of $\{(u_{m^{2 - \delta}} \times u_{n^{2 - \delta}}) \st m,n \in \NN\}\cdot x_0$ for any chosen $\eps>0$.

\begin{remark}
    The above argument only works for almost square times $(m^{2-\delta},n^{2-\delta})$ because of the following. 
    With the exploited techniques, it is not possible to approach a given point $\omega$ with sufficiently large periodic $U$-orbit, at distance of the order $R^{-1/2}$ within the scale $[D/4R,DR]^2$ of the $U$-orbit at $x_0$, and stay close to this orbit for a large enough time frame ensuring that there are enough elements of the form $(m^{2},n^{2})$ within that time frame. Indeed, although we can guarantee that if $\br\in[D/4R,DR]^2$ is chosen so that $u_{\br}x_0$ is close to $\omega$, then the relevant $U$-orbits stay close to each other on a time frame of size $R^{1/2}$, it is apriori not even clear whether the set
    \[ I = \{(m^2,n^2)\st m,n\in \ZZ\} \cap [r, r + R^{1/2}]^2 \]
    is empty or not. The presence of $\delta>0$ allows to build in enough margin in order to show that the corresponding set is not only not empty, but arbitrarily large, if $R$ is chosen large enough. Thus, the density statement about almost squares is precisely at the edge of the proof strategy carried out in this paper. 
\end{remark}

\subsection{Open Questions}
\subsubsection*{Uniform lattices} The most undesirable constraint in \Cref{thm:main} is the restriction to non-uniform lattices. Indeed, it should be easier for an infinite set to become $\eps$-dense in a compact quotient of $G$. Notice that the restriction to cocompact lattices in \cite{Ve10} provides further evidence for that. However, the lack of periodic horospherical orbits in such quotients renders the methods in the proof of \Cref{thm:main} irrelevant, as it crucially depends on the existence of $\eps$-dense periodic orbits.
\subsubsection*{Equidistribution} One direction of interest in extending \Cref{thm:main} would be to prove not just density of the set $\{(u_{m^{2 - \delta}} \times u_{n^{2 - \delta}}) \st m,n \in \NN\}\cdot x_0$, but equidistribution within $X$. This would be in accordance with the behavior of the horocycle flow evaluated at the sparse subset $\{n^{1+\eta}\st n\in \NN\}$ discussed in \cite{Ve10}.
\subsubsection*{Optimal explicit effective statement} Another direction of interest would be to obtain an explicit effective version of \Cref{thm:main}, that is, explicitly relating the density parameter $\eps$, or the size $L$ of the required periodic $U$-orbit, with the size of the range parameter $R$. Explicit and close to optimal effective equidistribution statements about large periodic $U$-orbits are known (\cite{KM96,Ve10}). We remark that one can extract a relation of the form $R^\eta = L$ for some $\eta>0$ from the proof of \Cref{thm:main}. However, this relation is presumably far from optimal.
\subsubsection*{Squares and (almost) primes} Finally, we address the problem of further sparsening the discrete sequence of evaluations of the $U$-orbit. We already discussed that the density of the set of squares $\{(u_{m^2},u_{n^2})\st m,n\in\NN\}.x_0$, is on the edge of the proof strategy and seems to be out of reach with the techniques currently available. Nevertheless, it would certainly be interesting to study. As in \cite{KR25} it would also be interesting to examine the set of primes $\{(u_{p},u_{q})\st p,q\text{ primes}\}\cdot x_0$ or almost primes $\{(u_{p_1p_2},u_{q_1q_2})\st p_1,p_2,q_1,q_2\text{ primes}\}\cdot x_0$.
    
\subsection{Notation}

    Throughout this paper, we use bold face letters (\eg $\br$, $\bR$, $\bs$, $\bt$) to denote vectors in $\RR^2$. We write $\br = (r_1,r_2)$ to denote the individual components of a vector in $\RR^2$. Moreover, for functions (\eg $\log$) we use the compressed notation $\log\br$ to denote the vector with components $(\log(r_1),\log(r_2))$.
    Moreover, for $\br,\bR\in\RR^2$ we will throughout use the interval notation $\rect[\br]{\bR}$ to denote the subset $[r_1,R_1]\times[r_2,R_2]$ of $\RR^2$.

    Denoting by $U,A,V$ the standard upper triangular, positive diagonal and lower triangular subgroups in $\SL_2(\RR)\times\SL_2(\RR)$, respectively, we use the notations
    \[ u_{\bl} = (u_{\ell_1},u_{\ell_2})\in U, \quad a_{\bl} = (a_{\ell_1},a_{\ell_2})\in A  \quad \text{and} \quad v_{\bl} = (v_{\ell_1},v_{\ell_2})\in V, \]
    for any $\bl = (\ell_1,\ell_2)\in\RR^2$.

    For $\beta>0$ we write $B_\beta$ for the $\beta$-box around the identity in $G=\SL_2(\RR)\times\SL_2(\RR)$ in the coordinates $G=UAV$, that is
    \[ B_\beta = \{g=u_{\br}a_{\bt}v_{\bs} \st \br,\bt,\bs\in [0,\beta]^2\} \subset \SL_2(\RR)\times\SL_2(\RR). \]
    We also write $B^L_\beta$ for the $\beta$-box around the identity in any subgroup $L\subseteq\SL_2(\RR)\times\SL_2(\RR)$ of the form $U,A,V, UA$ or $AV$.

    Finally, throughout the paper, especially in \Cref{sec:equi-exp-hor}, we use the notion of Sobolev norm $\Sob$. For any $f\in C^\infty_c(X)$ we put
    \[ \Sob(f) = \sum_{\operatorname{ord}(D)\leq \sobdim}\norm{Df(g)}[\infty], \]
    where the differential operator $D$ ranges over all monomials of order $\leq\sobdim$ in a basis of $\Lie(G) = \lieg$. 
    
    Properties of the Sobolev norm used in the paper include the following. For a fixed Riemannian metric $d(\cdot,\cdot)$ on $G$ and for $g$ in a fixed compact set, for any $f\in C^\infty_c(X)$ and $x\in X$, it holds that $\abs{f(gx)-f(x)} \ll \operatorname{Sob}_1(f)d(g,e)$. 
    For $\lie = \Lie(G)$, let $\norm{\cdot}[op]$ denote the operator norm of $\Ad\colon\lieg\to\lieg$. Another elementary property of Sobolev norms (\cite[Lemma 2.2]{Ve10}) is 
    \[ \Sob(f_1\cdot f_2)\ll_\sobdim \Sob(f_1)\Sob(f_2), \quad \Sob(g\cdot f_1) \ll_\sobdim \norm{g}[op]\Sob(f_1) \]
    for all $f_1,f_2\in C^\infty_c(X)$ and $g\in G$.
    %auto-file
\subsection*{Acknowledgments}
The starting point of this paper was the \highlight{Reading Groups in Analysis} workshop in the Department of Mathematics, University of Pennsylvania, held in August 4-8, 2025. The author would like to thank for Noy Soffer Aranov, Zuo Lin, Yuping Ruan, Jiyao Tang for the pleasant group atmosphere during discussions and for the collaboration in the initial phase of this project. 
Moreover, the author expresses his deep gratitude to Adam Kanigowski and Amir Mohammadi for their mentoring during the workshop. In particular, to Adam Kanigowski for outlying the proof of \cite{KR25} and encouraging us to look into the setting of the present paper. 
The stay at the workshop was sponsored by the U.S. National Science Foundation and the Department of Mathematics, University of Pennsylvania, to whom the author is sincerely thankful.
In addition, the author is grateful to Manfred Einsiedler for his continuous support and helpful discussions during the preparation of this paper.
	\section{Approaching periodic \texorpdfstring{$U$}{U}-orbits}

    Let $G = \SL_2(\RR)\times\SL_2(\RR)$ and define the three subgroups $U, A, V$ of $G$ by    
    \begin{align*}
        U &= \bigg\{(u_{r_1}, u_{r_2}) = \left(\begin{pmatrix}
            1 & r_1\\
             & 1
        \end{pmatrix}, \begin{pmatrix}
            1 & r_2\\
             & 1
        \end{pmatrix}\right)\st r_1,r_2\in\RR\bigg\}, \\
        A &= \bigg\{(a_{t_1}, a_{t_2}) = \left(\begin{pmatrix}
            e^{t_1/2} & \\
             &e^{-t_1/2}
        \end{pmatrix}, \begin{pmatrix}
            e^{t_2/2} & \\
             & e^{-t_2/2}
        \end{pmatrix}\right)\st t_1,t_2\in\RR\bigg\},\\
        V &= \bigg\{(v_{s_1}, v_{s_2}) = \left(\begin{pmatrix}
            1 & \\
            s_1 & 1
        \end{pmatrix}, \begin{pmatrix}
            1 & \\
            s_2 & 1
        \end{pmatrix}\right)\st s_1,s_2\in\RR\bigg\}.
    \end{align*}
    Notice that $U$ (resp. $V$) is the unstable horospherical subgroup for any element $a_\bt\in A$ with $\bt\in\RR_{\geq0}$ (resp. $-\bt\in\RR_{\geq0}$). We define the homomorphism $\varphi \colon U\to \RR^2$, $\varphi(u_{(r_1,r_2)}) = (r_1,r_2)$.
    Moreover, we note that $U\times A\times V \to G$, $(u,a,v)\to uav$ is a surjective homomorphism.
    
    Let $\Gamma \leq G$ be a non-uniform, irreducible lattice. We define $X = \tquot{G}{\Gamma}$ and denote by $\msr_{X}$ the $G$-invariant probability measure on $X$ induced by the Haar measure $\msr_G$ on $G$.

    By a theorem of \citeauthor{Se69} \cite{Se69}, also implied by Margulis arithmeticity theorem \cite{Ma91}, $\Gamma$ is \highlight{arithmetic}, \ie there exists a real quadratic extension $K/\QQ$ with Galois involution $\sigma$ and with ring of integers $\mathcal{O}_K$ such that $\Gamma$ and $\iota(SL_2(\mathcal{O}_K))$ are commensurable, where
    \[ \iota\colon \SL_2(K) \hookrightarrow \SL_2(\RR)\times\SL_2(\RR),~ q\mapsto (q,\sigma(q)) \]
    is the Minkowski embedding. Thus, we set $\mathbf{G} \defeq \Res{K}(\SL_2)$ to be an algebraic group defined over $\QQ$ so that $\mathbf{G}(\RR) \simeq G$, $\mathbf{G}(\QQ)$ is identified with $\SL_2(K)$ and $\Gamma$ is commensurable with $\mathbf{G}(\ZZ)$.
    
\subsection{Structure of periodic horospherical orbits}
    First, we discuss the structure of \highlight{periodic} $U$-orbits on $X$, \ie orbits $U\omega$, for $\omega=g\Gamma\in X$ such that $U\cap g\Gamma g^{-1}$ is a lattice in the two dimensional subgroup $U$. In this case, we define the lattice 
    \[ \Lambda_\omega \defeq \varphi(U\cap q\Gamma q^{-1}) \subseteq \RR^2. \]
    Further, we define the \highlight{size} $L_\omega$ of the periodic orbit a $U\omega$ as the covolume of the lattice $\Lambda_\omega$ in $\RR^2$, that is
    \[ L_\omega \defeq \operatorname{covol}(\Lambda_\omega). \]
    Finally, for any periodic orbit $U\omega$, we define the unique $U$-invariant probability measure $\mu_\omega$ on $U\omega$ as the pullback of the normalized Lebesgue measure on $\RR^2/\Lambda_\omega$. 

    We recall the structure of periodic $U$-orbits in $X$ and give a short proof.
    \begin{lemma}\label{lem:structure-of-U-orbits}
        Let $\omega\in X$ be given. Then there is a finite subset $\Xi\subset \mathbf{G}(\QQ)$ such that
        \[ U\omega \text{ is periodic} \quad\Leftrightarrow\quad  \text{there exists } a\in A, q\in\Xi \text{ such that } U\omega = Uaq\Gamma. \]
        In particular, $U$-periodic orbits exist and for all $(x,y)\in \Lambda_\omega\setminus\{(0,0)\}$ we have $xy\neq 0$.
    \end{lemma}
    \begin{proof}
        Let $\mathbf{B}\subset \SL_2$ denote the group of upper triangular matrices in $\SL_2$ and set $\mathbf{P} = \Res{K}(\mathbf{B})$. Then $\mathbf{P}$ is a minimal and maximal $\QQ$-parabolic subgroup of $\mathbf{G}$. By a theorem of Borel and Harish-Chandra, the (right) $\Gamma$ action on $\tlquot{\mathbf{P}(\QQ)}{\mathbf{G}(\QQ)}$ has finitely many orbits. Let $\Xi\subset \mathbf{G}(\QQ)$ be a set of representatives of this action.
    
        Let $U\omega = Ug\Gamma$ be a $U$-periodic orbit. Then by definition, $U\cap g\Gamma g^{-1}$ is a lattice in $U$. This is equivalent to $gUg^{-1} \cap \Gamma$ being a lattice in $gUg^{-1}$. By Borel's density theorem, $gUg^{-1} \cap \Gamma$ is Zariski dense in $gUg^{-1}$ and hence $gUg^{-1}$ is a defined over $\QQ$. This implies that the normalizer of $gUg^{-1}$ is a minimal and maximal $\QQ$-parabolic subgroup of $\mathbf{G}$. Since all $\QQ$-parabolic subgroups of $\mathbf{G}$ are conjugate to each other by elements in $\mathbf{G}(\QQ)$, there is an element $a\in A$, $q\in\Xi$ and $\gamma\in\Gamma$ such that $g = aq\gamma$. It follows that $U\omega = Ug\Gamma = Uaq\gamma\Gamma = Uaq\Gamma$.

        Let $a\in A$ and $q\in \Xi$. Since $a\in A$ normalizes $U$, the orbit $U\omega = Uaq\Gamma$ is periodic if and only if $Uq\Gamma$ is periodic. Since $q^{-1}Uq$ is a connected unipotent $\QQ$-subgroup, $qUq^{-1}\cap\Gamma$ is a lattice.

        Let now $(x,y)\in \Lambda_\omega\setminus\{(0,0)\}$ be given. Then $u_{(x,y)} = q\gamma q^{-1} \in U\cap q\Gamma q^{-1}$ for some $\gamma\in\Gamma$. Since $\gamma = (\gamma_0, \sigma(\gamma_0))$ and $q = (q_0,\sigma(q_0))$ we have
        \[ u_y = \sigma(q_0)\sigma(\gamma_0)\sigma(q_0)^{-1} = \sigma(q_0\gamma_0q_0^{-1}) = \sigma(u_x) = u_{\sigma(x)} \]
        and so $y = \sigma(x)$. It follows that $xy = x\sigma(x) = \operatorname{Nr}(x)\neq 0$ as $(x,y)\neq (0,0)$.
    \end{proof}

    \begin{definition}
        Let $\Lambda\subset\RR^2$ be a lattice. For $i=1,2$, we define the $i$-th successive minimum as
        \[ \lambda_i(\Lambda) \defeq\{\ell~\st \Lambda\text{ contains $i$ linearly independent vectors of $\norm{~\cdot~}[2]$-norm}\leq \ell\}. \]
    \end{definition}
    
    By Minkowski's theorems, we have for any lattice $\Lambda\subset\RR^2$ that
    \begin{align}\label{eq:succ-bounds}
        \frac{1}{\sqrt{2}}\lambda_1(\Lambda)^2 \leq \frac{1}{\sqrt{2}}\lambda_1(\Lambda)\lambda_2(\Lambda) \leq \operatorname{covol}(\Lambda) \leq \frac{\sqrt{3}}{2} \lambda_1(\Lambda)\lambda_2(\Lambda).
    \end{align}
    
    \begin{definition}
        A lattice $\Lambda\subset\RR^2$ is called \highlight{well-rounded}, if $\lambda_1(\Lambda) = \lambda_2(\Lambda)$. We call a periodic $U$-orbit $U\omega \subset X$ \highlight{well-rounded}, if $\Lambda_\omega\subseteq\RR^2$ is well-rounded.
    \end{definition}

    \begin{lemma}
        For any $L\geq 1$, there exist well-rounded periodic $U$-orbits $U\omega$ with $L_\omega \geq L$.
    \end{lemma}
    \begin{proof}
        Let $U\omega = Uq\Gamma$ be a periodic $U$-orbit. Let $a = a(\tau,t) = (a_{\tau+t},a_{\tau -t}) \in A$ for $\tau,t\in\RR$ and set $\omega' \defeq aq\Gamma$. Then 
        \[ U\cap aq\Gamma q^{-1}a^{-1} = \{ (u_{e^{\tau+t}s_1}, u_{e^{\tau -t}s_2}) \in U \st (u_{s_1},u_{s_2}) \in U\cap q\Gamma q^{-1} \} \] 
        and so it is easy to see that
        \[ L_{\omega'} = \operatorname{covol}\left(U\cap aq\Gamma q^{-1}a^{-1}\right) = e^{2\tau}\operatorname{covol}\left(U\cap q\Gamma q^{-1}\right). \]
        Thus, fixing $\tau\in\RR$ sufficiently large, we get $L_{\omega'} \geq L$. Moreover, notice that the covolume of the lattice $U\cap aq\Gamma q^{-1}a^{-1}$ does not depend on $t$.

        What remains to show is that one can find $t\in\RR$ such that $\Lambda_{\omega'}$ is well-rounded.
        We define the family of lattices $\Lambda_t \defeq \varphi(U\cap a(\tau,t)q\Gamma q^{-1}a(\tau,t)^{-1})\subset\RR^2$ for $t\in\RR$. Let $u\in \Lambda_0$ be such that
        \[ \lambda_1(\Lambda_0)^2 = \norm{u}[2]^{2} = e^{2\tau}(r_1^2 + r_2^2). \]
        for $(r_1,r_2)\in\RR^2$ with $u_{(r_1,r_2)}\in U\cap q\Gamma q^{-1}$. We define the continuous function 
        \[ f_u\colon\RR_{\geq0}\to\RR,\quad f_u(t) = \norm{a(\tau,t)ua(\tau,t)^{-1}}[2]^2 = e^{2\tau}(e^{2t}r_1^2 + e^{-2t}r_2^2). \]
        By \Cref{lem:structure-of-U-orbits} we have $r_1\neq 0$ so $f$ is unbounded and moreover $a(\tau,t)ua(\tau,t)^{-1}\in\Lambda_t$ for all $t\in\RR$. Let
        \[ t_0 = \sup\{t\in\RR\st f_u(t) = \lambda_1(\Lambda_{t})\}. \]
        We claim that $t_0<\infty$ and that $\Lambda_{t_0}$ is a well-rounded lattice, so that $Ua(\tau,t_0)\omega$ is as required.

        Indeed, $f_u$ is unbounded and for large enough $T\in\RR_{\geq0}$, $\rest{f}{[T,\infty)}$ is strictly increasing. Thus, using \eqref{eq:succ-bounds} we obtain $t_0<\infty$. For any $t > t_0$, $f_u(t) > \lambda_1(\Lambda_t)$ and thus there exists $u'_t \in\Lambda_t$ with $\norm{u'_t}[2] < \norm{a(\tau,t)ua(\tau,t)^{-1}}[2]$. Now taking any accumulation point $u'$ of $u't$ with $t\searrow t_0$, we get by discreteness of $\Lambda_t$ for all $t$, that
        \[ u' \neq a(\tau,t_0)ua(\tau,t_0)^{-1} = \lim_{t\searrow t_0}a(\tau,t)ua(\tau,t)^{-1} \]
        and by continuity of $f$, that
        \[ \norm{u'}[2] = \lambda_1(\Lambda_{t_0}) = f(t_0) = \norm{a(\tau,t_0)ua(\tau,t_0)^{-1}}[2]. \]
        Hence, $v$ and $a(\tau,t_0)ua(\tau,t_0)^{-1}$ span $\Lambda_{t_0}$ and $\Lambda_{t_0}$ is well-rounded.  
    \end{proof}
    
\subsection{Quantitative non-divergence of horospherical subgroups}\label{sec:quan-non-div}
    Let 
    \[ \inj(x) = \sup\{\eps \in\RR^+\st g\mapsto gx\text{ is injective on } B^G_\eps\} \]
    denote the \highlight{injectivity radius} on $X$. For every $\eps > 0$ let $X_\eps =\{x\in X \st \inj(x) \geq \eps \}$ be the set of lattices with injectivity radius at least $\eps$. By Mahler's compactness criterion, any compact subset of $X$ is contained in $X_\eps$ for some $\eps>0$. \Cref{lem:quan-non-div} gives a quantitative estimate on how much time an expanding piece of the horospherical orbit $Vy$ for $y\in X$ can spend outside of the set $X_\eps$. Quantitative non-divergence results for polynomial maps 
    where first proved by \citeauthor{KM98} in \cite{KM98} using the notion of $(C,\alpha)$-good functions. We adapt the proof of \cite[Proposition 3.1]{LM23} to the present setting of expanding pieces of the \highlight{full horospherical} subgroup $V$. The proof ultimately relies on the non-divergence results of Margulis, Dani and Kleinbock.
    
    \begin{lemma}\label{lem:quan-non-div}
        There exists an absolute constant $C_0>0$ so that the following holds.\\
        Let $0<\eta,\eps<1$ and $y\in X$. Let $I\subseteq[0,1]^2$ be a rectangle of side lengths at least $\eta$. Then
        \[ \Leb(\{\bs\in I\st a_{-\log\bt}v_{\bs}y \not\in X_{\eps}\}) \leq C_0\eps^{1/2}\Leb(I), \]
        as long as $\bt\in\RR_{\geq 1}^2$ satisfies
        \begin{align*}
            \log t_1+ \log t_2\geq 2\abs{\log(\eta^2\inj(y))} + C_0.
        \end{align*}
    \end{lemma}

    Since the statement in \Cref{lem:quan-non-div} is insensitive to passing from $\Gamma$ to the commensurable lattice $\mathbf{G}(\ZZ)$, increasing the absolute constant $C_0$ otherwise, we may (and will) for ease of exposition assume that $\Gamma = \mathbf{G}(\ZZ) \simeq \SL_2(O_K)$. 
    Let $\mathfrak{g} = \Lie(G) = \sl_2(\RR)\oplus\sl_2(\RR)$ be equipped with the $\QQ$-structure 
    \[ \mathfrak{g}_\QQ = \sl_2(K)\subset \mathfrak{g}. \]
    We denote by $\norm{\cdot}$ the maximum norm on $\Mat[2](\RR)\times\Mat[2](\RR)$ with respect to the standard basis. Let $p_1,p_2$ denote the projections to the first and second component of $\Mat[2](\RR)\times\Mat[2](\RR)$, respectively. Notice that $\mathcal{O}_K^\times\lieg_\ZZ = \lieg_\ZZ$. An elementary observation (see \eg \cite[Lemma 8.6]{KT07}) gives the existence of some $c\in\RR$ so that for every $\liew=(\liew_1,\liew_2)\in\lieg$ with $\liew_1\neq 0$ or $\liew_2\neq 0$ there exists $\theta\in \mathcal{O}_K^\times$ such that for $i=1,2$
    \begin{align}\label{eq:comp-norm}
        c^{-1}(\norm{\liew_1}\norm{\liew_2})^{1/2} \leq \norm{p_i(\theta\liew)} \leq c(\norm{\liew_1}\norm{\liew_2})^{1/2}.
    \end{align}
    Moreover, we fix a basis $\{\liev_1,\liev_2\}$ of $\Lie(V)$ consisting of primitive integral vectors as follows. Write $K = \QQ(\sqrt{d})$ and set $\liev_1 = (E_{21},E_{21})$ and $\liev_2 = (\sqrt{d}E_{21},-\sqrt{d}E_{21})$ where $E_{21} = \smat{0}{0}{1}{0}$. We further define
    \[ \liev \defeq \liev_1 \wedge \liev_2 \in \wedge^2\lieg. \]
    Notice that $\liev\in \wedge^2\lieg_\ZZ$, so that for any $g\in\mathbf{G}(\QQ)$ the set $\Gamma g.\liev$ is contained in the subset rational vectors in $\wedge^2\lieg_\QQ$ whose denominators are bounded in terms of $g$. In particular, $\Gamma g.\liev$ is a discrete and closed subset of $\wedge^2\lieg_\QQ$. For any $g=(g_1,g_2)\in G$ we then have
    \begin{align}\label{eq:act-on-liev}
        g.\liev = (g.\liev_1)\wedge(g.\liev_2) =-2\sqrt{d}(g_1.E_{21},0)\wedge(0,g_2.E_{21}).
    \end{align}
    In particular, the induced maximum norm $\norm{\cdot}$ on $\wedge^2\lieg$ with respect to the standard basis then gives 
    \begin{align}\label{eq:norm-act-on-liev}
        \norm{g.\liev} = 2\norm{p_1(g.\liev_1)}\norm{p_2(g.\liev_2)},
    \end{align}
    where $p_i$ denotes the projection to the $i$-th component in $\Mat[2](\RR)\times\Mat[2](\RR)$. 
    
    Recall that $\Xi\subset\mathbf{G}(\QQ)$ is the finite set of representatives for the $\Gamma$-action on $\tlquot{\mathbf{P}(\QQ)}{\mathbf{G}(\QQ)}$, so that
    \begin{align}\label{eq:cusps}
         \mathbf{G}(\QQ) = \mathbf{P}(\QQ)\Xi\Gamma.
    \end{align}
    Equivalently, $\Xi$ parametrizes the finite set of different cusps in $X$. 
    We define $\alpha\colon X\to[2,\infty)$ by
    \[ \alpha(g\Gamma) = \max\{\norm{g\gamma\xi.\liev}^{-1}\colon\xi\in\Xi^{-1},\gamma\in\Gamma\}. \]
    
    \begin{lemma}\label{lem:help}
        Using the notation from above, we have the following.
        \begin{enumerate}
            \item\label{eq:itm-1}(Unique Cusp) There exists $C_1 = C_1(\Gamma)\geq2$ so that the following holds.\\
            Let $g\Gamma\in X$. If $\alpha(g\Gamma)\geq C_1$, then there is $\xi_0\in\Xi^{-1}$ and $\gamma_0\in\Gamma$ so that $\norm{g\gamma_0\xi_0.\liev}^{-1} = \alpha(g\Gamma)$ and 
            \[ \norm{g\gamma\xi.\liev} > 1/C_1, \quad \text{for all } (\xi,\gamma) \text{ so that } \gamma\xi.\liev \neq \gamma_0\xi_0.\liev. \]
            \item\label{eq:itm-2} (Short vectors) There exists $C_2$ so that the following holds.\\
            Let $0<\rho,\eta < 1$, $t_1,t_2>0$, and $g\in G$. Let $I=I_1\times I_2\subseteq[0,1]^2$ be a rectangle of side lengths at least $\eta$. Then
            \[ \Leb\left(\{s\in I\st\norm{(a_{-t_1},a_{-t_2})v_sg.\liev}\leq e^{t_1+t_2}\rho^4\eta^4\norm{g.\liev}\}\right) \leq C_2\rho\Leb(I). \]\label{eq:max-escape}
        \end{enumerate}
    \end{lemma}
    \begin{proof}
        \eqref{eq:itm-1}: For a detailed proof we refer to \cite[Lemma A.1(1)]{LM23}. There is $\eps>0$ so that the following holds. Suppose there exist $\gamma,\gamma'\in\Gamma$, and $\xi,\xi'\in\Xi^{-1}$ so that $\gamma\xi.\liev \neq \gamma'\xi'.\liev$ and
        \begin{align}\label{eq:assumption}
            \norm{g\gamma\xi.\liev} < \eps \text{ and } \norm{g\gamma'\xi'.\liev}<\eps.
        \end{align}
        Then one can show that the four elements
        \[ w_1 = g\gamma\xi. \liev_1, \quad w_2 = g\gamma\xi. \liev_2, \quad w_1' = g\gamma'\xi'. \liev_1, \quad w_2' = g\gamma'\xi'. \liev_2, \]
        form a nilpotent subalgebra in $\lieg$ of dimension at least $3$. However, as the dimension of any maximal nilpotent subalgebra in $\lieg$ is $2$, this is a contradiction to \eqref{eq:assumption}.

        \eqref{eq:itm-2}:
        For every $g\in G$ and $\rho>0$, set
        \[ I(g,\rho) = \left\{s\in I\st \norm{p_i^{(21)}(v_sg.\liev_i)} \leq \frac{1}{64}\rho\eta^2\norm{p_i(g.\liev_i)} \text{ for } i=1 \text{ or } i = 2\right\}, \]
        where $p_i^{(21)}$ denotes the projection onto $\RR E_{21}$ inside the $i$-th component. We claim that for some absolute constant $C>0$
        \begin{align}\label{eq:size-of-exc-set}
            \Leb(I(g,\rho)) \leq 4C\rho^{1/2}\Leb(I).
        \end{align}
        Indeed, since $g.\liev_1 = (g_1.E_{21},g_2.E_{21})$ and $g.\liev_2 = (\sqrt{d}g_1.E_{21},-\sqrt{d}g_2.E_{21})$ it holds that
        \begin{align*}
            I(g,\rho) \subseteq &\left\{s_1\in I_1\st \norm{p_1^{(21)}(v_{s_1}.E_{21})} \leq \frac{1}{64}\rho\eta^2\norm{g_1.E_{21}}\right\}\times I_2\\
            &\cup I_1\times\left\{s_2\in I_2\st \norm{p_2^{(21)}(v_{s_2}g_2.E_{21})} \leq \frac{1}{64}\rho\eta^2\norm{g_2.E_{21}}\right\}.
        \end{align*}
        A direct computation gives
        \[ p^{(21)}_i(v_\bs.\lie w) = -s_i^2 w_i^{(12)} +2s_i w_i^{(11)}+ w_i^{(21)} \]
        for any $\liew =(\liew_1,\liew_2)\in\lieg$ with $\liew_i = \smat{w_i^{(11)}}{w_i^{(12)}}{w_i^{(21)}}{w_i^{(22)}}\in\sl_2(\RR)$ and $v_\bs\in V$. Since the side lengths of the rectangle $I=I_1\times I_2\subseteq[0,1]^2$ are by assumption at least of size $\eta < 1$, the Remez inequality together with bounds on the Chebyshev polynomials give
        \[ \sup_{s\in I_i}\abs{p_i^{(21)}(v_s.\lie w)} \geq \tfrac{1}{64}\abs{I_i}^2\sup_{s\in [0,1]^2} \abs{p_i^{(21)}(v_s.\lie w)}  \geq \tfrac{1}{64}\eta^2\norm{\lie w_i}. \]
        By \cite[\S 3]{KM98}, the function $p^{(21)}_i(v_{s_i} g_i.E_{21})$ is $(4\sqrt{3},1/2)$-good and so there exists $C>0$ such that
        \[ \abs{\left\{s_i\in I_i\st \norm{p_i^{(21)}(v_{s_i}g_i.E_{21})} \leq \frac{1}{64}\rho\eta^2\norm{g_i.E_{21}}\right\}}[\bigg] \leq 2C\rho^{1/2}\abs{I_i}. \]
        Hence, we get
        \[ \Leb(I(g,\rho)) \leq  2C\rho^{1/2}\abs{I_1}\cdot\abs{I_2} + \abs{I_1}\cdot2C\rho^{1/2}\abs{I_2} \leq 4C\rho^{1/2}\Leb(I). \]

        Let now $s\in I\setminus I(g,64\rho^2)$, then
        \[ \norm{p_i^{(21)}(v_sg.\liev_i)} > \rho^2\eta^2\norm{p_i(g.\liev_i)} \text{ for } i = 1,2. \]
        Since $a_{-\bt}.\lie w = (e^{t_1}\lie w_1, e^{t_2}\lie w_2)$ for all $\bt= (t_1,t_2)$ and $\lie w=(\lie w_1,\lie w_2) \in\operatorname{span}\left\{(E_{21},0), (0,E_{21})\right\}$ we obtain, using \eqref{eq:norm-act-on-liev}, that
        \begin{align*}
             \norm{a_{-\bt}v_\bs g.\liev} &\geq \norm{(p_1^{(21)}(a_{-\bt}v_{\bs} g.\liev_1),0)\wedge(0,p_2^{(21)}(a_{-\bt}v_\bs g.\liev_2))}\\
             &=e^{t_1+t_2}\norm{(p_1^{(21)}(v_{\bs} g.\liev_1),0)\wedge(0,p_2^{(21)}(v_\bs g.\liev_2))}\\
             &\geq 2e^{t_1+t_2}\norm{p_1^{(21)}(v_{\bs}g.\liev_1)}\cdot\norm{p_2^{(21)}(v_{\bs}g.\liev_2)}\\
             &> 2e^{t_1+t_2}\rho^4\eta^4\norm{p_1(g.\liev_1)}\cdot\norm{p_2(g.\liev_2)} = e^{t_1+t_2}\rho^4\eta^4\norm{g.\liev}.
        \end{align*}
        If we set $C_2\defeq 32C$, we obtain
        \[ \Leb\left(\{s\in I\st\norm{(a_{-t_1},a_{-t_2})v_sg.\liev}\leq e^{t_1+t_2}\rho^4\eta^4\norm{g.\liev}\}\right) \leq C_2\rho\Leb(I), \]
        since $\Leb(I(g,64\rho^2)) \leq 32C\rho\Leb(I)$ by \eqref{eq:size-of-exc-set}.
    \end{proof}

    The following \Cref{lem:inj-rel-height} relates the injectivity radius at $x\in X$ to the height function $\alpha$ at $x$.
    
    \begin{lemma}[{\cite[Lemma A.2]{LM23}}]\label{lem:inj-rel-height}
        There exists $C_3>0$ so that
        \[ C_3^{-1}\alpha(x)^{-1} \leq \inj(x)^2 \leq C_3\alpha(x)^{-1} \]
        holds for all $x\in X$.
    \end{lemma}
    \begin{proof}
        We first discuss the upper bound on $\inj(x)^2$. By reduction theory for arithmetic groups (see \eg \cite{PR94}), there exist $T$, $S$ so that
        \[ (\SO_2(\RR)\times\SO_2(\RR))\cdot\{a_{\bt}\st t_1 + t_2 \leq T\}\cdot\{v_\bs\st s_1 + s_2 \leq S\} \cdot \Xi \]
        is a fundamental domain for $\Gamma$ in $G$. Let $x = g\Gamma$, then by \Cref{lem:help}\eqref{eq:itm-1} there exists $\gamma_0\in\Gamma$, and $\xi_0\in\Xi^{-1}$ so that if we write $g\gamma_0\xi_0 = ka_{\bt} v_\bs$ with $t_1+t_2 \leq T$ and $k\in\SO_2(\RR)\times\SO_2(\RR)$, then
        \[ \alpha(x) = \max\{\norm{g\gamma\xi.\liev}^{-1}\st \xi\in\Xi^{-1},\gamma \Gamma\} = \norm{g\gamma_0\xi_0.\liev}^{-1} = \norm{ka_{\bt}v_\bs.\liev}^{-1} \geq  \frac{1}{2}e^{t_1+t_2}\norm{\liev}^{-1}, \] 
        where we used that the operator norm of $k$ is lower bounded by $\frac{1}{2}$.
        
        Let $g\gamma_0\xi_0.\liev = 2\sqrt{d}(\liew_1,0)\wedge(0,\liew_2)$ using \eqref{eq:act-on-liev} and notice that then $g\gamma_0\xi_0.\liev_1 = (\liew_1,\liew_2)$. By \eqref{eq:norm-act-on-liev}, there is $\theta\in\mathcal{O}_K^\times$ so that for $\liew \defeq (\liew_1,\liew_2)\in\lieg$  and $i=1,2$ we have 
        \[ c^{-1}(\norm{\liew_1}\norm{\liew_2})^{1/2} \leq\norm{\theta\liew_i}\leq  c(\norm{\liew_1}\norm{\liew_2})^{1/2}. \]
        However, then
        \[ \norm{g\gamma_0\xi_0.\theta\liev_1} = \norm{(\theta\liew_1,\theta\liew_2)}
        \leq c(\norm{\liew_1}\norm{\liew_2})^{1/2} = \frac{c}{(2\sqrt{d})^{1/2}}\norm{g\gamma_0\xi_0.\liev}^{1/2}\ll e^{-(t_1+t_2)/2}, \] 
        and so there is an element $v_\bs \in g\gamma_0\xi_0(V\cap \Gamma) \xi_0^{-1}\gamma_0^{-1}g^{-1}$ of size $\ll e^{-(t_1+t_2)/2}$.
        
        Notice that $g\Gamma g^{-1}$ and $g\gamma_0\xi_0\Gamma \xi_0^{-1}\gamma_0^{-1}g^{-1}$ are commensurable. Thus, there is $M\in\NN$ (depending only on $\Xi$) so that
        \[ v_{\bs}^M = v_{M\bs} \in g\Gamma g^{-1} \]
        and $v_{M\bs}$ is of size $\ll Me^{-(t_1+t_2)/2}$. For $C>0$ large enough depending only on $\Gamma$, we thus get that $\inj(x)^2 \leq C\alpha(x)^{-1}$ holds for all $x\in X$.
        
        We turn to the proof of the lower bound for $\inj(x)^2$. Let $x=g\Gamma$ and assume that $\inj(x)< \eps$.~Then
        \[ g\Gamma g^{-1}\cap B^G_\eps \neq \{e\}. \]
        If $\eps$ is small enough, $g\Gamma g^{-1}$ consists only of unipotent elements, and so there exists some nilpotent $\liew\in\lieg_\ZZ$ so that
        \[ \norm{g.\liew} \ll \eps \]
        where the implied constant is absolute. Since all minimal $\QQ$-parabolic subgroups of $\mathbf{G}$ are $\mathbf{G}(\QQ)$-conjugate to each other, and since $\mathbf{G}(\QQ) = \mathbf{P}(\QQ)\Xi\Gamma$ (see \eqref{eq:cusps}), there are $\xi\in\Xi^{-1}$ and $\gamma\in\Gamma$ so that $\liew \in \gamma\xi.\Lie(V)$, that is
        \[ \liew = \gamma\xi.\left((a+b\sqrt{d})E_{12}, (a-b\sqrt{d})E_{12}\right) \]
        for some $a,b\in \frac{1}{N}\ZZ$ with $(a,b)\neq(0,0)$, where $N$ only depends on $\Xi$. By the Iwasawa decomposition of $\mathbf{G}$, we can write $g\gamma\xi = ka_{-\bt} v$, with $k\in\SO_2(\RR)\times\SO_2(\RR)$, $a_{-\bt} \in A$ and $v\in V$, so that
        \[ \norm{g.\liew} = \norm{ka_{-\bt} v.\left((a+b\sqrt{d})E_{12}, (a-b\sqrt{d})E_{12}\right)}[{}][\Big] \leq 2 \max\{e^{t_1}\abs{a+b\sqrt{d}},e^{t_2}\abs{a-b\sqrt{d}}\} \ll \eps, \]
        where we used that the operator norm of $k$ is upper bounded by $2$. In particular, this implies that $e^{t_1}\abs{a+b\sqrt{d}}\cdot e^{t_2}\abs{a-b\sqrt{d}} = e^{t_1+t_2}(a^2 - b^2d) \ll \eps^2$.
        However, as $a,b\in\frac{1}{N}\ZZ$ are not both $0$ and $d\in\NN$ is not a square, we have $a^2-b^2d \geq \frac{1}{N^2}$ and so $e^{t_1+t_2}\ll_N \eps^2$ follows. Therefore, we obtain
        \[ \norm{g\gamma\xi.\liev} = 2\norm{p_1(ka_{-\bt}.\liev_1)}\norm{p_2(ka_{-\bt}.\liev_2)} \ll 2\sqrt{d}e^{t_1+t_2} \leq C'\eps^2 \]
        for some absolute constant $C'$ depending on $\Gamma$. Since $\alpha(x)^{-1} \leq \norm{g\gamma\xi.\liev}$, the lower bound follows.
        Setting $C_3 = \max\{C,C'\}$ finishes the proof.
    \end{proof}

    \begin{proof}[Proof of \Cref{lem:quan-non-div}]
        If we choose $C>0$ large enough, then by \Cref{lem:inj-rel-height} we obtain that if $\log t_1+\log t_2\geq 2\abs{\log(\eta^2\inj(y)))} + C$, then $\log t_1+\log t_2 \geq \log(\alpha(y)/\eta^4)$ .

        Let $y = g\Gamma \in X$ and set $\rho \defeq \frac{1}{2}C_2^{-1}$. In view of \Cref{lem:help}\eqref{eq:max-escape} we have
        \[ \sup\left\{\norm{a_{-\log \bt}v_{\bs}g\gamma\xi.\liev}\st \bs\in I\right\} \geq e^{\log t_1 + \log t_2}\rho^4\eta^4\norm{g\gamma\xi.\liev} \geq \rho^4e^{\log t_1 + \log t_2}\eta^4\alpha(y)^{-1}. \]
        for all $\gamma\in\Gamma$ and $\xi\in\Xi^{-1}$. As long as $\log t_1+\log t_2 \geq \log(\alpha(y)/\eta^4)$, the right most term above is clearly $\geq \rho^4$.
        Hence, under these assumptions on $t_1$ and $t_2$ we have the lower bound $\sup\left\{\alpha(a_{-\log \bt}v_{\bs}y)^{-1}\st \bs\in I\right\} \geq~\rho^4$ and so we can apply \cite[Theorem 4.1]{KM98} to obtain
        \begin{align}\label{eq:bound-using-alpha}
            \Leb(\{\bs\in I\st \alpha(a_{-\log \bt}v_{\bs}y)^{-1} < C_3^{-1}\eps^2 \}) \leq C'\eps^{1/2}\Leb(I) 
        \end{align}
        for some $C'>0$.
        \Cref{lem:inj-rel-height} yields
        \[ \{\bs\in I\st \inj(a_{-\log\bt}v_{\bs}y) < \eps\} \subseteq \{\bs\in I\st \alpha(a_{-\log\bt}v_{\bs}y)^{-1} < C_3^{-1}\eps^2\}\]
        and so by setting $C_0 = \max\{C,C'\}$ the desired result follows from \eqref{eq:bound-using-alpha}.
    \end{proof}
    
\subsection{\texorpdfstring{$A$}{A}-recurrence}
    For $x\in X$ with dense $U$-orbit we have the following recurrence property of the $A$-orbit of $x$.
    \begin{lemma}\label{lem:A-recurrence}
        Let $x = g\Gamma\in X$ with dense $U$-orbit. There is a compact set $K\subseteq X$ such that the set of return times
        \[ \{\bt\in\RR_{\geq 1}^2\st a_{-\log\bt}x \in K \} \]
        is unbounded.
    \end{lemma}
    \begin{proof}
        Recall the notion $\Lie(G) =\lieg$ of \Cref{sec:quan-non-div}. We fix a basis $\{\lieu_1,\lieu_2\}$ of $\Lie(U)$ consisting of primitive integral vectors as follows. Let $E_{12} = \smat{0}{1}{0}{0}$ and define $\lieu_1 = (E_{12},E_{12})$, $\lieu_2 = (\sqrt{d}E_{12},-\sqrt{d}E_{12})$ and $\lieu \defeq \lieu_1\wedge\lieu_2 \in \wedge^2\lieg$.
        Similarly as \eqref{eq:act-on-liev} and \eqref{eq:norm-act-on-liev}, we have
        \begin{align}\label{eq:help}
            g.\lieu = (g.\lieu_1)\wedge(g.\lieu_2) = -2\sqrt{d}(g_1.E_{12},0)\wedge(0,g_2.E_{12}), \quad \norm{g.\lieu} = 2\norm{p_1(g.\lieu_1)}\norm{p_2(g.\lieu_2)}.
        \end{align}
        Recall that $\mathbf P=\Res{K}(\mathbf B)$ where $\mathbf B$ is the subgroup of upper triangular matrices, $\mathbb P$ is the stabilizer of $
        \lieu$. Moreover, $\lieu$ and $\liev$ conjugated by a Weyl element, \ie
        \[ \left(\smat{}{1}{-1}{},\smat{}{1}{-1}{}\right).\liev = \lieu. \]
        Since $\left(\smat{}{1}{-1}{},\smat{}{1}{-1}{}\right) \in \mathbf{G}(\QQ)$ we have
        \[ \{\gamma\xi.\liev \st \gamma \in \Gamma, \xi \in \Xi^{-1} \} = \{\gamma\xi.\lieu \st \gamma \in \Gamma, \xi \in \Xi^{-1} \} \]
        and so
        \[ \alpha(g\Gamma) = \max\{\norm{g\gamma\xi.\liev}^{-1}\colon\xi\in\Xi^{-1},\gamma\in\Gamma\}. = \max\{\norm{g\gamma\xi.\lieu}^{-1}\colon\xi\in\Xi^{-1},\gamma\in\Gamma\}. \]

        Let $K = \{g\Gamma\in X\st \alpha(g\Gamma) \leq 2C_1\}$ where $C_1$ is the constant from \Cref{lem:help}. Using the relation between $\alpha$ and the injectivity radius from \Cref{lem:inj-rel-height} and Mahler's compactness criterion, we deduce that $K$ is a compact set. We claim that if $x\in X$ has a dense $U$-orbit, then $\{\bt\in\RR_{\geq 1}^2\st a_{-\log\bt}x \in K \}$ is unbounded. 
        By \Cref{lem:structure-of-U-orbits} and the decomposition $\mathbf G(\QQ) = \mathbf P(\QQ)\Xi\Gamma$, a point $g\Gamma$ has a periodic $U$-orbit if and only if there are $\gamma\in\Gamma$ and $\xi\in \Xi^{-1}$ such that $\lieu$ is an eigenvector of $g\gamma\xi$. Since
        \[ g\gamma\xi.\lieu = (g\gamma\xi.\lieu_1)\wedge (g\gamma\xi.\lieu_2) =  -2\sqrt{d}((g\gamma\xi)_1.E_{12},0)\wedge(0,(g\gamma\xi)_2.E_{12}), \]
        $\lieu$ is an eigenvector of $g\gamma\xi$ if and only if $p_1^{(21)}(g\gamma\xi .\lieu_1)= 0$ and $p_2^{(21)}(g\gamma\xi.\lieu_2) = 0$, where we recall that $p_i^{(21)}$ denotes the projection to $\RR E_{21}$ in the $i$-th component.

        Assume that $x=g\Gamma\not\in K$. Then by \Cref{lem:help} there is $\xi_0\in\Xi^{-1}$ and $\gamma_0\in\Gamma$ such that
        \[ \alpha(x) = \norm{g\gamma_0\xi_0.\lieu}^{-1} > 3C_1 \]
        and $\norm{\gamma\xi.\lieu} > 1/C_1$ for all $\gamma\in\Gamma$, $\xi\in\Xi^{-1}$ with $\gamma\xi.\lieu \neq \gamma_0\xi_0.\lieu$. Since $\lieu$ is not an eigenvector of $g\gamma_0\xi_0$, we can choose $i\in\{1,2\}$ such that $w^{(21)} \defeq p_i^{(21)}(g\gamma_0\xi_0.\lieu_i) \neq 0$. For ease of notation, we assume that $i=1$. By \eqref{eq:help}, and since $p_1^{(21)}(a_{(-\log t,0)}g\gamma\xi.\lieu_1) = e^{t}w^{(21)}$ we have
        \[ \norm{a_{(-\log t,0)}g\gamma_0\xi_0.\lieu} \geq 2e^{t}w^{(21)}\norm{p_2(g\gamma_0\xi_0.\lieu_2)} > 0  \]
        for all $t\geq 1$. For $t_0 = 1$ we have $\frac{1}{3C_1} > \norm{a_{(-\log t_0,0)}g\gamma_0\xi_0.\lieu} > 2w^{(21)}\norm{p_2(g\gamma_0\xi_0.\lieu_2)}$ and so choosing $t_0>1$ large enough, one can by continuity ensure that 
        \[ \frac{1}{C_1} > \norm{a_{(-\log t_0,0)}g\gamma_0\xi_0.\lieu} = 2e^{t_0}w^{(21)}\norm{p_2(g\gamma_0\xi_0.\lieu_2)}= \frac{1}{2C_1} > \frac{1}{3C_1}. \]
        Invoking \Cref{lem:help} once again, we get
        \[ \alpha(a_{(-\log t_0,0)}x) = \norm{a_{(-\log t_0,0)}g\gamma_0\xi_0.\lieu}^{-1} = 2C_1 < 3C_1 \]
        and so $a_{(-\log t_0,0)}x \in K$.

        To sum up, if $x\in X$ has a dense $U$-orbit and $x\not\in K$, then there is $t_0 > 1$, such that either $a_{(-\log t_0,0)}x \in K$ or $a_{(0, -\log t_0)}x \in K$. Since $t_0>1$ and since for all $\bt\in\RR^2$ the $U$-orbit of $a_{\bt}x\in X$ is dense, it follows that
        \[ \{\bt\in\RR_{\geq 1}^2\st a_{-\log\bt}x \in K\} \]
        is unbounded.
    \end{proof}
    
\subsection{Equidistribution of expanding horospherical orbits}\label{sec:equi-exp-hor}

    We recall that expanding pieces of (unstable) horospherical orbits become equidistributed in $X$, with an effective error rate.
    \begin{lemma}[{\cite[Proposition 2.4.8]{KM96}}]\label{lem:exp-horospheres}
        Let $x\in X$ with dense $U$-orbit. There exists $\kappa>0$ and $\sobdim\in\NN$ so that for a large enough constant $C_x>0$, depending on $x$, the following holds.\\
        For all $C\geq C_x$, $y\in K$ and all $f\in C^\infty_c(X)$ we have
        \begin{align}\label{eq:exp-horo}
            \abs{\frac{1}{4}\int_{[1/2,1]^2}f((a_{\log C},a_{\log C})u_{\br}y)\dd\br - \int_Xf\msr_X}[\bigg] \leq \Sob(f) C^{-\kappa}.
        \end{align}
    \end{lemma}
    The constant $\kappa$ in \Cref{lem:exp-horospheres} is determined by the spectral gap of $X = \tquot{G}{\Gamma}$. 
    The result is classical and the method of proof goes back to the thesis of Margulis. The idea is to exploit exponential mixing of the flow $\{\tilde{a}_t \defeq a_{(t,t)}\st t\geq1\}$ on $X$, see \cite[Proposition 2.4.8]{KM96}. 
    \section{Shadowing periodic \texorpdfstring{$U$}{U}-orbits}
    The goal of this section is to prove that points $x\in X$ with dense $U$-orbits approach a given closed $U$-orbit $U\omega$ \highlight{effectively}. More precisely, \Cref{prop:approaching-omega} states that at an infinite number of times $\bR$ the $U$-orbit of $x$, is sufficiently close to the starting point $\omega$ such that the distance is roughly of the order $\bR^{-1/2+\delta/10}$. This implies that the $U$-orbit \enquote{shadows} the periodic $U$-orbit $\omega$ for times up to size $\bR^{1/2-\delta/5}$ (see \Cref{cor:staying-close-to-closed-orbits}).
    
\subsection{Approaching periodic \texorpdfstring{$U$}{U}-orbits}

    Let $\omega \in X$ with periodic $U$-orbit $U\omega$, and $x\in X$ with dense $U$-orbit be given.
    The following \namecref{prop:approaching-omega} asserts that the $U$-orbit at $x$ approaches the closed orbit $U\omega$ infinitely often (at times $(\br^k)_{k\in\NN}$) and effectively, that is, the distance between the two orbits is controlled by $\bR^k$ (where $\br^k\sim\bR^k)$. Due to technical reasons which will become apparent later (see \Cref{lem:weyl-sum-bound} and \Cref{lem:weyls-inequality}) we need to obtain an upper and a lower bound on the distance to the periodic $U$-orbit. 

    \begin{proposition}\label{prop:approaching-omega}
        Let $x\in X$ with dense $U$-orbit, $\omega\in X$ with periodic $U$-orbit and $1\geq \delta > 0$. There exists $D_x>0$ and an infinite sequence $(\bR^k)_{k\in\NN} = (R^k_1,R^k_2)_{k\in\NN}$, with $R^k_i\to\infty$ for $i=1,2$, such that there is 
        \[ \br^k\in \rect[\tfrac{D_x}{4}R^k_1]{D_xR^k_1}\times\rect[\tfrac{4D_x}{4}R^k_2]{4D_xR^k_2} \]
        satisfying
         \[ u_{\br^k}x \in \left\{a_{\bt}\st \bt\in\rect{2\log(1 + (\bR^k)^{-\frac{\delta}{12(\sobdim+3)}})}\right\}\cdot \left\{v_{\bs}\st \bs\in\rect[\tfrac{1}{2}(\bR^k)^{-1/2+\delta/10}]{(\bR^k)^{-1/2+\delta/10}}\right\}\cdot\omega. \]
    \end{proposition}
    \begin{proof}
        Fix $\sigma_0 > 1$ such that $4\log\sigma_0 \geq C_0 - 2\log(\inj(\omega))$, where $C_0$ is as in \Cref{lem:quan-non-div}. By \Cref{lem:A-recurrence}, we may pick an unbounded sequence $(\bt^k)_{k\in\NN}$, such that $t_1^k,t_2^k\to\infty$ as $k\to\infty$, satisfying $a_{-\log\bt^k}x\in K$. Let $\kappa$, $\sobdim$ and $C_x$ be as in \Cref{lem:exp-horospheres}. Let $D_x>0$ be a sufficiently large parameter specified below. Define $\bR^k = D_x^{-1}\bt^k$ for all $k\in\NN$. As $\bt^k$ is unbounded, $\bR^k$ is also unbounded. Hence, we may choose $k_0\in\NN$ large enough and remove the first $k_0$ elements of $(\bR^k)_{k\in\NN}$, to ensure that $(R^k_1R^k_2)^{1/2-\delta/24} > \sigma_0^2\inj(\omega)$ and $\frac{1}{2}(\bR^k)^{-1/2 +\delta/10} > (R^k_1R^k_2)^{-1/2+\delta/24}\sigma_0$  holds for all $k\in\NN$. The first condition guarantees that
        \begin{align*}
            2\abs{\log\left((R^k_1R^k_2)^{-1/2+\delta/24}\sigma_0^2 \inj(\omega)\right)}[\bigg] + C_0 &= \log((R^k_1R^k_2)^{1-\delta/12}) - 4\log \sigma_0 - 2\log(\inj(\omega)) + C_0\\
            &\leq \log((R^k_1)^{1-\delta/12}) + \log((R^k_2)^{1-\delta/12})
        \end{align*}
        holds for all $k\in\NN$, where the last inequality follows by the choice of $\sigma_0$. Thus, this allows us to apply \Cref{lem:quan-non-div} for $y=\omega$ with 
        \[ \eta  = (R^k_1R^k_2)^{-1/2+\delta/24}\sigma_0,\quad \eps = \tfrac{1}{2}, \quad I = \rect[\tfrac{1}{2}{(\bR^k)^{-1/2+\delta/10}}]{(\bR^k)^{-1/2+\delta/10}}, \quad \bt = (\bR^k)^{1-\delta/12}, \]
        since the side lengths of $I$ are at least $\eta$ by the second condition. It follows that there is $\bs_0\in\rect[\tfrac{1}{2}{(\bR^k)^{-1/2+\delta/10}}]{(\bR^k)^{-1/2+\delta/10}}$ such that
        \[ z_0 \defeq a_{-\log(\bR^k)^{1-\delta/12}} v_{\bs_0}\omega \in X_{1/2}. \]

        Recall that for $\beta>0$ we have $B_\beta = B_\beta^G = \{u_ra_tv_s\st r,t,s\in[0,\beta]\}$.
        Let $\chi_{z_0}$ be an appropriately chosen smooth function satisfying (see \cite[Proposition 2.4.7(b)]{KM96})
        \[ \mathbbm{1}_{B_{\beta/2} z_0} \leq \chi_{z_0} \leq \mathbbm{1}_{B_{\beta} z_0}, \quad \Sob(\chi_{z_0}) \ll \beta^{3-\sobdim} \]
        where $\beta>0$ depends on $k$ and is defined as
        \[ \beta = \log(1+ (\bR^k)^{-\frac{\delta}{12(\sobdim+3)}}). \]
        As $\bR^k$ is unbounded, we may assume that $\beta < 1/2$. 
        For all $0\leq x\leq 1$ we have $x/2\leq \log(1+x) \leq x$, so $\beta \sim (\bR^k)^{-\frac{\delta}{12(\sobdim+3)}}$.
        Hence, with this choice of $\beta$ we have
        \begin{align}\label{eq:properties-of-beta}
            \int_X\chi_{z_0}\msr_X \sim \msr_G(B_\beta) = \beta^6 \sim (\bR^k)^{-\frac{\delta}{2(\sobdim+3)}} \quad\text{and}\quad \Sob(\chi_{z_0}) \ll \beta^{3-\sobdim} \sim (\bR^k)^{\frac{(\sobdim-3)\delta}{12(\sobdim+3)}}. 
        \end{align}      
        
        Applying \Cref{lem:exp-horospheres} with $C = D_x(\bR^k)^{\delta/12}$, $y = a_{-\log\bt^k}x \in K$ and $f = \chi_{z_0}$ yields
        \begin{align}\label{eq:appl-eff-equi}
            \abs{\frac{1}{4}\int_{[1/2,1]^2}\chi_{z_0}((a_{\log C},a_{\log C})u_{\br}a_{-\log\bt^k}x)\dd\br - \int_X\chi_{z_0}\msr_X}[\bigg] \ll (\bR^k)^{\frac{(\sobdim-3)\delta}{12(\sobdim+3)}}C^{-\kappa},
        \end{align}
        and we assume $k\in\NN$ is large enough so that $C\geq C_x$.
        The right hand side above is
        \[ (\bR^k)^{\frac{(\sobdim-3)\delta}{12(\sobdim+3)}}C^{-\kappa} = D_x^{-\kappa}(\bR^k)^{-\frac{\delta}{2(\sobdim+3)}}, \]
        while $\int_X\chi_{z_0}\msr_X = O((\bR^k)^{-\frac{\delta}{2(\sobdim+3)}})$ by \eqref{eq:properties-of-beta}. We choose $D_x>0$ large enough, so that $D_x^{-\kappa}$ is smaller than the (explicit) constant in $\int_X\chi_{z_0}\msr_X = O((\bR^k)^{-\frac{\delta}{2(\sobdim+3)}})$. By this choice of $D_x$,
        for all $k$ large enough, we can find $\br\in [1/2,1]^2$ such that        
        \begin{align}\label{eq:box-zo}
            (a_{\log C},a_{\log C})u_{\br}a_{-\log\bt^k}x = u_{\br'}\cdot\{a_{\bt}\st \bt\in [0,\beta]\}\cdot \{v_{\bs}\st \bs\in[0,\beta]\}\cdot z_0 \in B_\beta z_0
        \end{align}
        for some $\br'\in [0,\beta]$. Using that $(a_{\log C},a_{\log C})u_{\br} = u_{C\br}(a_{\log C},a_{\log C})$ and that
        \[ (a_{\log C},a_{\log C})a_{-\log\bt^k} = a_{-\log(\bR^k)^{1-\delta/12}} \] 
        we obtain, by multiplying both sides in \eqref{eq:box-zo} with $u_{-\br'}$, that
        \[ u_{C\br-\br'} a_{-\log(\bR^k)^{1-\delta/12}}x \in \{a_{\bt}\st \bt\in [0,\beta]\}\cdot \{v_{\bs}\st \bs\in [0,\beta]\}\cdot a_{-\log(\bR^k)^{1-\delta/12}} v_{\bs_0}\omega. \]
        This is equivalent to
        \[ a_{-\log(\bR^k)^{1-\delta/12}}u_{(C\br - \br')(\bR^k)^{1-\delta/12}}x \in a_{-\log(\bR^k)^{1-\delta/12}}\cdot\{a_{\bt}\st \bt\in [0,\beta]\}\cdot \{v_{\bs}\st \bs\in [0,\beta(\bR^k)^{-1+\delta/12}]\}\cdot v_{\bs_0}\omega. \]    
        Multiplying by $a_{-\log(\bR^k)^{1-\delta/12}}$ on both sides we obtain
        \[ u_{\br^k} x = a_{\log\bt'} v_{\bs'} \omega, \]
        where $\br^k \defeq (C\br - \br')(\bR^k)^{1-\delta/12}= D_x\bR^k\br - (\bR^k)^{1-\delta/12}\br'$ and
        \[ \log\bt'\in [0,\log(1+(\bR^k)^{-\frac{\delta}{12(\sobdim+3)}})],\quad  \bs'\in [\tfrac{1}{2}(\bR^k)^{-1/2+\delta/10},\beta(\bR^k)^{-1+\delta/12} + (\bR^k)^{-1/2+\delta/10}]. \]
        Considering the bounds on $\br\in[1/2,1]^2$ and $\br'\in[0,\beta]$, we can find $k_0'\in\NN$ large enough so that for all $k \geq k_0'$
        \[ \br^k = D_x\bR^k\br - (\bR^k)^{1-\delta/12}\br' \in \rect[\tfrac{D_x}{4}R^k_1]{D_xR^k_1}\times\rect[\tfrac{D_x}{4}R^k_2]{D_xR^k_2} , \]
        as well as
        \[ 2(\bR^k)^{-1/2 + \delta/10} \geq (\bR^k)^{-\frac{\delta}{12(\sobdim+3)}-1 +\delta/12} + (\bR^k)^{-1/2+\delta/10} \geq \beta(\bR^k)^{-1+\delta/12} + (\bR^k)^{-1/2+\delta/10}, \]
        where the first inequality holds since for $\sobdim \geq 2$ 
        \[ -1/2 + \delta/10 \geq -\frac{\delta}{12(\sobdim+3)}-1 +\delta/12. \]
        Upon removing the first $k_0'$ elements of $(\bR^k)_{k\in\NN}$, this completes the proof of the \namecref{prop:approaching-omega}.
    \end{proof}
   
\subsection{Setup and Preliminaries} 
    Let $x\in X$ with dense $U$-orbit and $\omega\in X$ with periodic $U$-orbit and $1\geq \delta > 0$ be given. Let $(\bR^k)_{k\in\NN}$ and $(\br^k)_{k\in\NN}$ be as in \Cref{prop:approaching-omega}. We define
    \begin{align}\label{eq:def-x'}
         x' \defeq u_{\br^k}x = a_{\log \bt'}v_{\bs'}\omega
    \end{align}
    and recall that $\log\bt' \in \rect{2\log(1 + (\bR^k)^{-\frac{\delta}{12(\sobdim+3)}})}$ and $\bs'\in \rect[\frac{1}{2}(\bR^k)^{-1/2+\delta/10}]{(\bR^k)^{-1/2+\delta/10}}$.
     
    The following \namecref{cor:staying-close-to-closed-orbits} to \Cref{prop:approaching-omega} gives a quantitative control of the distance between $u_{\br}x'$ and the $U$-periodic orbit $U\omega$ for a time range of $\br$ up to size $(\bR^k)^{1/2 -\delta/5}$. Since the $U$-orbits of $x'$ and $\omega$ will not stay parallel in general, (unless $a_{\log \bt'}v_{\bs'}$ is in the centralizer of $U$), we need to reparameterize the speed of the $U$-orbit of $\omega$. For $\bt,\bs\in\RR^2_{>0}$ we define 
    \begin{align}\label{eq:def-reparam}
        \reparam[\bt,\bs](\br) \defeq 
        %\frac{\br}{\bt+\bs\br} = 
        \left(\frac{r_1}{t_1 + s_1r_1},\frac{r_2}{t_2 + s_2r_2}\right).
    \end{align}

    \begin{corollary}\label{cor:staying-close-to-closed-orbits}
        Let $0<\eps<1$. Then, there exists $k_0\in\NN$ such that for all $k\geq k_0$ the following two properties hold:
        \begin{enumerate}
            \item For all $\br\in \rect[\br^k]{\br^k + (\bR^k)^{1/2-\delta/5}}$ we have
            \begin{align}\label{eq:eps-for-all-r}
                d(u_{\br}x, u_{\reparam(\br-\br^k)}\omega) \leq \eps^{1/2}.
            \end{align}
        \item  For any given $\br'\in \rect[\br^k]{\br^k + (\bR^k)^{1/2-\delta/5}}$ there exists $\omega''\in U\omega$ such that for all\hfill\newline $\br~\in~\rect[\br']{\br'+(\bR^k)^{1/2-\delta/5}}$ we have
            \begin{align}\label{eq:eps-for-given-r'}
                d(u_{\br}x, u_{\reparam[\bt'',\bs''](\br-\br')}\omega'') \leq \eps^{1/3},
            \end{align}
            where $\log\bt''\in \rect{2\log(1+3(\bR^k)^{-\frac{\delta}{12(\sobdim+3)}})}$ and $\bs''\in\rect[\frac{1}{2}(\bR^k)^{-1/2+\delta/10}]{8(\bR^k)^{-1/2+\delta/10}}$. 
        \end{enumerate}
    \end{corollary}
    \begin{remark}
        One should think of $\br'$ to be of the form $(n^{2-\delta},m^{2-\delta})$ in the corollary above.
    \end{remark}
    \begin{proof}
        Notice first that for all $r,t,s\in\RR_{>0}$
        \begin{align}\label{eq:first-obs}
            \begin{pmatrix}1&r\\ & 1\end{pmatrix}\begin{pmatrix}t^{1/2}&\\st^{-1/2}& t^{-1/2}\end{pmatrix} = \begin{pmatrix}t^{1/2}+rst^{-1/2}&\\st^{-1/2} & (t^{1/2}+rst^{-1/2})^{-1}\end{pmatrix}\begin{pmatrix}1&\tfrac{r}{t+sr}\\ & 1\end{pmatrix}.
        \end{align}
        Using $x' = a_{\log {\bt'}}v_{\bs'}\omega$ and the above in both components yields
        \begin{align}\label{eq:ur_x'}
            u_{\br'}x' = a_{2\log(\bt'^{1/2}+\br'\bs'\bt'^{-1/2})}v_{\bs'(1 + \br'\bs'\bt'^{-1})}u_{\reparam(\br')}\omega.
        \end{align}
        Thus, if $\br$ is small enough, the $U$-orbits of $x'$ and $\omega$ stay close together -- but note the difference in the speeds of the $U$-orbits.
        Indeed, for $\br\in \rect{(\bR^k)^{1/2-\delta/5}}$ we have, using $\bt'\geq1$, that
        \begin{equation}\label{eq:bound-on-a_t}
            \begin{aligned}
                \abs{2\log(\bt'^{1/2}+\br'\bs'\bt'^{-1/2})} &= \abs{2\log(\bt'^{1/2}(1+\br'\bs'\bt'^{-1}))}\\
                &\leq \log(\bt') + 2\log(1+\br'\bs')\\
                &\leq \log(1 + (\bR^k)^{-\frac{\delta}{12(\sobdim+3)}}) + 2\log(1+ (\bR^k)^{1/2-\delta/5}(\bR^k)^{-1/2+\delta/10})\\
                &\leq \log(1 + (\bR^k)^{-\frac{\delta}{12(\sobdim+3)}}) + 2\log(1+(\bR^k)^{-\delta/10})\\
                &< 2\log(1+3(\bR^k)^{-\frac{\delta}{12(\sobdim+3)}}),
            \end{aligned}
        \end{equation}
        where the last inequality holds for $k$ sufficiently large. Using \eqref{eq:bound-on-a_t}, $\bt'\geq1$ and the bounds on $\bs'$, we further obtain
        \begin{equation}\label{eq:bound-on-v_s} 
            \begin{aligned}
                \abs{\bs'} \leq \abs{\bs'(1+\br'\bs'\bt'^{-1})} &= \abs{\bs'\bt'^{-1/2}(\bt'^{1/2}+\br'\bs'\bt'^{-1/2})}\\
                &\!\!\overset{\eqref{eq:bound-on-a_t}}{\leq} 2\abs{\bs'\bt'^{-1/2}}(1+3(\bR^k)^{-\frac{\delta}{12(\sobdim+3)}})\\
                &\leq 2(\bR^k)^{-1/2+\delta/10}(1+3(\bR^k)^{-\frac{\delta}{12(\sobdim+3)}})\\
                &\leq 8(\bR^k)^{-1/2+\delta/10},
            \end{aligned}
        \end{equation}
        where the last inequality follows from $\bR^k\geq 1$. 

        Let now $0<\eps<1$ be given. Then, \eqref{eq:eps-for-all-r} follows since $(\bR^k)_{k\in\NN} = (R^k_1,R^k_2)_{k\in\NN}$, is a sequence with $R^k_i\to\infty$, for $i=1,2$, and since the metric $d$ is right-invariant. Indeed, if we define $\br'\defeq(\br - \br^k)\in \rect{(\bR^k)^{1/2-\delta/5}}$, then by \eqref{eq:ur_x'} we have
        \begin{align*}
            d(u_{\br}x, u_{\reparam(\br-\br^k)}\omega) &= d(u_{\br'}x', u_{\reparam(\br')}\omega)\\
            &= d(a_{2\log(\bt'^{1/2}+\br'\bs'\bt'^{-1/2})}v_{\bs'(1 + \br'\bs'\bt'^{-1})}u_{\reparam(\br')}\omega, u_{\reparam(\br')}\omega)\\
            &= d(a_{2\log(\bt'^{1/2}+\br'\bs'\bt'^{-1/2})}v_{\bs'(1 + \br'\bs'\bt'^{-1})},e) \leq \eps^{1/2},
        \end{align*}
        where the last inequality holds for $k$ large enough, by \eqref{eq:bound-on-a_t} and \eqref{eq:bound-on-v_s}.
        
        Finally, for any $\br'\in\rect[\br^k]{\br^k + (\bR^k)^{1/2-\delta/5}}$ set $\omega'' \defeq u_{\reparam(\br'-\br^k)}\omega$. We have
        \begin{align*}
            u_{\br'}x = u_{\br' -\br^k}x' &= a_{2\log(\bt'^{1/2}+(\br'- \br^k)\bs'\bt'^{-1/2})}v_{\bs(1 + (\br'- \br^k)\bs'\bt'^{-1})}u_{\reparam(\br'-\br^k)}\omega\\
            &= a_{2\log(\bt'^{1/2}+(\br'- \br^k)\bs'\bt'^{-1/2})}v_{\bs'(1 + (\br'- \br^k)\bs'\bt'^{-1})}\omega''.
        \end{align*}
        Hence, $u_{\br'}x = a_{\log\bt''}v_{\bs''}\omega''$, for $\bt'' = (\bt'^{1/2}+(\br'-\br^k)\bs'\bt'^{-1/2})^2$ and $\bs'' = \bs'(1 + (\br'- \br^k)\bs'\bt'^{-1})$, and the necessary bounds for $\log\bt''$ and $\bs''$ follow from $\br'-\br^k\in\rect{(\bR^k)^{1/2-\delta/5}}$ and the calculations in \eqref{eq:bound-on-a_t} and \eqref{eq:bound-on-v_s}. We thus have
        \[ u_{\br}x = u_{\br-\br'}u_{\br'}x  = u_{\br-\br'}a_{\log\bt''}v_{\bs''}\omega'' = \left(u_{\br-\br'}a_{\log\bt''}v_{\bs''}u_{-\reparam[\bt'',\bs''](\br-\br')}\right)u_{\reparam[\bt'',\bs''](\br-\br')}\omega''. \]
        Using the observation in \eqref{eq:first-obs} and the bounds on $\bt''$, $\bs''$ and $\br-\br'$, we can choose $k_0\in\NN$ large enough so that $d(u_{\br-\br'}a_{\log\bt''}v_{\bs''}u_{-\reparam[\bt'',\bs''](\br-\br')},e)\leq\eps^{1/3}$ for all $k\geq k_0$.
    \end{proof} 
    \section{Discrete equidistribution in periodic \texorpdfstring{$U$}{U}-orbits}

    In this section we will pass from the continuous setting to the discrete setting. We will prove \Cref{prop:equi-at-almost-squares} which roughly says that if a periodic $U$-orbit is considered at a certain, large enough discrete set of times (a reparametrization of a large set of almost squares), then this set becomes equidistributed inside the $U$-orbit.

    We will continue using the notations and assumptions of the previous sections. That is, $x\in U$ has a dense $U$-orbit, $\omega\in X$ is a $U$-periodic orbit and $1\geq \delta > 0$ is fixed. Recall the definition of $\reparam[\bt,\bs]$ for $\bt,\bs\in\RR_{>0}^2$ given in \eqref{eq:def-reparam}. 

    Throughout this section, we set $A = \frac{1}{12(\sobdim+3)}$.

    \begin{proposition}\label{prop:equi-at-almost-squares}
        Let $\omega\in X$ have a well-rounded, periodic $U$-orbit. Let $f\in\Lip(X)$ be fixed with compact support and $C>c>0$. Let $\bR=(R_1,R_2)$ and let $\br\in\rect[cR_1]{CR_1}\times\rect[cR_2]{CR_2}$,  $\bt'\in[0,2\log(1+3\bR^{-A\delta})]$ and $\bs'\in\rect[\frac{1}{2}\bR^{-1/2+\delta/10}]{8\bR^{-1/2+\delta/10}}$ be given. Then there exists $\br' \in \rect[\br]{\br + \bR^{1/2-\delta/5}}$ such that
        \begin{align*}
            \sup_{\omega'\in U\omega}\abs{\frac{1}{\#I}\sum_{(m,n)\in I}f(u_{\reparam[\bt'',\bs''](m^{2-\delta}-r'_1,n^{2-\delta}-r'_2)}\omega ') - \int_{U\omega}f\dd\mu_{\omega}}[\bigg] = o(\norm{\bR^{-1}}[\infty]),
        \end{align*}
        where $\bt''=(\bt'^{1/2}+(\br'-\br)\bs'\bt'^{-1/2})^2$ and $\bs'' = \bs'(1+(\br'-\br)\bs'\bt'^{-1})$ and
        \[ I = \ZZ^2\cap \rect[\br'^{1/(2-\delta)}]{(\br' + \bR^{1/2 -\delta/5})^{1/(2-\delta)}}. \] 
    \end{proposition}

    The proof of \Cref{thm:main} is straightforward assuming \Cref{prop:equi-at-almost-squares}, \Cref{prop:approaching-omega} and \Cref{cor:staying-close-to-closed-orbits}. 
    
    \begin{proof}[Proof of \Cref{thm:main}]
        Assume that $\{u_{(m^{2-\delta},n^{2-\delta})}x\st m,n\in\NN\}$ is not dense in $X$. Then there exists $\rho>0$ and $z\in X$ such that
        \[ \{u_{(m^{2-\delta},n^{2-\delta})}x\st m,n\in\NN\}\cap B_\rho z = \emptyset. \]
        Let $f$ be a non-negative Lipschitz function satisfying $\mathbbm{1}_{B_{\rho/2} z} \leq f \leq \mathbbm{1}_{B_\rho z}$. Then it clearly holds that $\int_X f\dd\msr_X > \msr_X\left(B_{\rho/2} z\right) > 0$, whereas for any $I\subseteq\NN^2$ we have 
        \begin{align}\label{eq:sum-is-zero}
            \sum_{(m,n)\in I} f(u_{(m^{2-\delta},n^{2-\delta})}x) = 0. 
        \end{align}

        By \cite{KM96, Ve10}, large, periodic $U$-orbits equidistribute in $X$. Thus, we can pick $L$ big enough such that for all periodic $U$-orbits $U\omega$ of size $L_{\omega} \geq L$ we have
        \[ \int_{U\omega}f\dd\mu_{\omega} > \frac{1}{2}\msr_X\left(B_{\rho/2} z\right). \]
        Let $(\bR^k)_{k\in\NN}$ and $(\br^k)_{k\in\NN}$ be the sequences given by \Cref{prop:approaching-omega} applied to $x$ and a point $\omega\in X$ with a well-rounded periodic $U$-orbit of size $L_\omega \geq L$. We apply \Cref{prop:equi-at-almost-squares} to $f$, $c=\frac{D_x}{4}$, $C=D_x$ (with $D_x$ as in \Cref{prop:approaching-omega}), $\bR = \bR^k$ and $\br = \br^k$, where we choose $k\in\NN$ big enough so that there exists $\br' \in \rect[\br^k]{\br^k + (\bR^k)^{1/2-\delta/5}}$ satisfying
        \begin{align*}
            \sup_{\omega'\in U\omega}\abs{\frac{1}{\#I}\sum_{(m,n)\in I}f(u_{\reparam[\bt'',\bs''](m^{2-\delta}-r'_1,n^{2-\delta}-r'_2)}\omega ') - \int_{U\omega}f\dd\mu_{\omega}}[\bigg] < \frac{1}{2}\msr_X\left(B_{\rho/2} z\right)
        \end{align*}
        where $I = \ZZ^2\cap\rect[\br'^{1/(2-\delta)}]{(\br' + (\bR^k)^{1/2-\delta/5})^{1/(2-\delta)}}$.
        This implies that
        \begin{align}\label{eq:positive-at-almost-squares}
            \sum_{(m,n)\in I}f(u_{\reparam[\bt'',\bs''](m^{2-\delta}-r'_1,n^{2-\delta}-r'_2)}\omega ') > \# I\left(\int_{U\omega}f\dd\mu_{\omega} - \frac{1}{2}\msr_X\left(B_{\rho/2} z\right)\right) > 0 \quad \text{for all } \omega'\in U\omega.
        \end{align}
        Now we use \Cref{cor:staying-close-to-closed-orbits}, more precisely \eqref{eq:eps-for-given-r'}, for all $\br=(m^{2-\delta},n^{2-\delta})$ with $(m,n)\in I$, to deduce that for $k$ large enough there exists $\omega''\in U\omega$ with
        \begin{align}\label{eq:final-estimate}
            \sum_{(m,n)\in I} f(u_{(m^{2-\delta},n^{2-\delta})}x) \geq \sum_{(m,n)\in I}\left(f(u_{\reparam[\bt'',\bs''](m^{2-\delta}-r'_1,n^{2-\delta}-r'_2)}\omega'') + O_f(\eps^{1/3})\right).
        \end{align}
        If we choose $\eps$ small enough so that the right hand side in \eqref{eq:final-estimate} is positive, which is possible by \eqref{eq:positive-at-almost-squares}, the left hand side is positive as well. This is a contradiction to \eqref{eq:sum-is-zero}.
    \end{proof}

\subsection{Proof of \texorpdfstring{\Cref{prop:equi-at-almost-squares}}{prop:equi-at-almost-squares}}

    The proof idea is as follows. We will reduce the statement in the proposition to an effective equidistribution statement in the torus $\TT^2$. More precisely, the result will follow by proving that the evaluations of a Taylor approximation of $\reparam[\bt,\bs]$ given in \Cref{lem:taylor} at a large set will be sufficiently dense inside $\TT^2$. We will achieve this using an effective version of Weyl's inequality (see \Cref{lem:weyls-inequality}).
     
    Recall that 
    \[ \reparam[\bt,\bs](\br) = 
    %\left(\frac{\br}{\bt + \br\bs}\right) = 
    \left(\frac{r_1}{t_1 + s_1r_1},\frac{r_2}{t_2 + s_2r_2}\right),\]
    for $\bt,\bs\in\RR^2_{>0}$. We write $\reparam[t_i,s_i](r_i)$ for the $i$-th component of $\reparam[\bt,\bs](\br)$, for $i=1,2$. We will first approximate $\reparam[\bt,\bs]$ by its Taylor polynomial. 
    \begin{lemma}[Taylor expansion of ${\reparam[\bt,\bs]}$]\label{lem:taylor} For $i=1,2$ let $R_i> 1$ and $J_i = \rect{R_i^{1/2-\delta/5}}$. Assume that $\log t_i\in\rect{2\log(1+3R_i^{-A\delta})}$ and $s_i \in \rect[\frac{1}{2}R_i^{-1/2+\delta/10}]{8R_i^{-1/2+\delta/10}}$. Define $c_{i,\ell} \defeq (-1)^{\ell}{s_i}^{\ell-1}{t_i}^{-\ell}$ for $0\leq \ell \leq d_i\defeq \max\{\lceil 5\delta^{-1}\rceil,2\}$ and $p_{t_i,s_i}(r) \defeq \sum_{\ell=0}^{d_i}c_{i,\ell}r^\ell$.
    Then for all $r_i\in J_i$ we have,
    \[ \abs{\reparam[t_i,s_i](r_i) - p_{t_i,s_i}(r_i)} = O\left((\log R_i)^{-1}\right). \]
    \end{lemma}
    \begin{proof}
        The proof for both $i=1,2$ is the same, so we drop all subscripts $i$ from the notations. First, notice that the $\ell$-th derivative of $\reparam[t,s]$ is
        \[ \reparam[t,s]^{(\ell)}(r) = (-1)^\ell\ell!\frac{ts^{\ell-1}}{(t+ rs)^{\ell+1}}. \]
        Moreover, by the bounds on $t\in[1,(1+3R^{-A\delta})^2]$, $s$ and $r$ we have $\frac{t}{(t+rs)^{\ell+1}}\leq (1+3R^{-A\delta})^2$. Hence,
        \begin{align}\label{eq:taylor-bound}
            \abs{\frac{\reparam[t,s]^{(\ell)}(r)}{\ell!}(R^{1/2-\delta/5})^\ell}[\bigg] =  \abs{\frac{ts^{\ell-1}}{(t+ rs')^{\ell+1}}}[\bigg](R^{1/2-\delta/5})^\ell \ll R^{(\ell-1)(-1/2+\delta/10) + (1/2-\delta/5)\ell}(1+3R^{-A\delta})^2.
        \end{align} 
        For $\ell = d+1$ the exponent of $R$ in the main term
        \[ (\ell-1)(-1/2+\delta/10) + (1/2-\delta/5)\ell = (1/2 -\delta/10) -\ell\delta/10, \]
        is negative by the choice of $d$. Hence, by Taylor expansion up to order $d\geq 2$ we obtain
        \[ \reparam[t,s](r) = \sum_{\ell=0}^d\frac{\reparam[t,s](0)}{\ell!}r^\ell + O(r^{d+1}) =  \sum_{\ell=0}^d\frac{(-1)^{\ell}s^{\ell-1}}{t^\ell}r^\ell +  O(r^{d+1}) = p_{t,s}(r) + O(r^{d+1}),  \] 
        and the error term $O(r^{d+1})$ is of order $O((\log R)^{-1})$ by \eqref{eq:taylor-bound}.
    \end{proof}

    \begin{proof}[Proof of \Cref{prop:equi-at-almost-squares}]
        The goal of the proof is to find times $\br' \in \rect[\br]{\br+\bR^{1/2-\delta/5}}$ such that $u_{P_{\br'}(m,n)\omega'}$ is sufficiently dense in $U\omega$, when $(m,n)$ ranges over elements in 
        \[ I = \ZZ^2\cap \rect[\br'^{1/(2-\delta)}]{(\br' + \bR^{1/2 -\delta/5})^{1/(2-\delta)}} \]
        and $P_{\br'} \defeq (p_{t_1,s_1},p_{t_2,s_2})$ consists of the polynomials provided by \Cref{lem:taylor}.\\
        
        For $R>1$ and $r\in[cR,CR]$ consider the interval $J \defeq [(r^{1/(2-\delta)}, (r + R^{1/2-\delta/5})^{1/(2-\delta)}]$. If $R$ is large enough, then $\NN\cap J\neq\emptyset$.
        This follows immediately from the mean value theorem applied to the function $x\to x^{1/(2-\delta)}$ on the interval $[cR,(C+1)R]\supseteq[r,r + R^{1/2-\delta/5}]$. Indeed, it implies that
        \begin{align}\label{eq:bigger-than-1}
            R^{\tfrac{\delta + 2\delta^2}{10(2-\delta)}} \gg \frac{1}{2-\delta}(cR)^{\tfrac{1}{2-\delta}-1}R^{1/2-\delta/5} \geq \abs{J} \geq \frac{1}{2-\delta}((C+1)R)^{\tfrac{1}{2-\delta}-1}R^{1/2-\delta/5} \gg R^{\tfrac{\delta + 2\delta^2}{10(2-\delta)}}.
        \end{align}
        If $R$ is large enough, then $\abs{J}\geq 1$ and so $\NN\cap J\neq \emptyset$.
            
        Let now
        \begin{align}\label{eq:interval}
            J_1 \times J_2 = [(cR_1)^{1/(2-\delta)}, (CR_1 + R_1^{1/2-\delta/5})^{1/(2-\delta)}]\times[(cR_2)^{1/(2-\delta)}, (CR_2 + R_2^{1/2-\delta/5})^{1/(2-\delta)}],
        \end{align}
        where $R_i$ is large enough so that $(r'_1)^{1/(2-\delta)}\in \NN\cap J_1$ and $(r'_2)^{1/(2-\delta)}\in \NN\cap J_2$ for some $r_1,r_2'\in\RR$. 
        Define $\br' = (r'_1,r'_2)$ and notice that $r'_i\in[cR_i, CR_i + R_i^{1/2-\delta/5}]\subseteq[cR_i,(C+1)R_i]$, for $i=1,2$. 
        
        Set $\bR = (R_1,R_2)$ and assume $\log\bt'\in \rect{2\log(1+3\bR^{-A\delta})}$ and $\bs'\in \rect[\frac{1}{2}\bR^{-1/2+\delta/10}]{8\bR^{-1/2+\delta/10}}$ are given. Let $P_{\br'} = (p_{t'_1,s'_1},p_{t'_2,s'_2})$, where $p_{t'_1,s'_1}$ and $p_{t'_2,s'_2}$ are the polynomials defined in \Cref{lem:taylor}, that is $p_{t'_i,s'_i}(r) = \sum_{\ell=0}^{d_i}c_{i,\ell}r^\ell$, where $c_{i,\ell} = (-1)^{\ell}{s'_i}^{\ell-1}{t'_i}^{-\ell}$ for $0\leq \ell \leq d_i= \max\{\lceil 5\delta^{-1}\rceil,2\}$, for $i=1,2$. Further, set
        \begin{align}\label{eq:integer-points}
            I \defeq I_1\times I_2 = \ZZ^2\cap\left[{r'}_1^{1/(2-\delta)}, (r'_1 + R_1^{1/2-\delta/5})^{1/(2-\delta)}\right]\times\left[{r'}_2^{1/(2-\delta)}, (r'_2 + R_2^{1/2-\delta/5})^{1/(2-\delta)}\right].
        \end{align}
        By the same argument as above, using the mean value theorem for $x \mapsto x^{1/(2-\delta)}$ on the intervals $[cR_1, (C+2)R_1]$ and $[cR_2, (C+2)R_2]$, we obtain
        \begin{align}\label{eq:bound-number-of-points}
            (R_1R_2)^{\tfrac{\delta + 2\delta^2}{10(2-\delta)}}  \gg \#I \gg (R_1R_2)^{\tfrac{\delta + 2\delta^2}{10(2-\delta)}}.
        \end{align}
        Observe that if $(m,n)\in I$, then
        \begin{align}\label{eq:int-close-to}
            \abs{m^{2-\delta} - r'_1} \leq R_1^{1/2-\delta/5} \quad\text{and}\quad \abs{n^{2-\delta} - r'_2} \leq R_2^{1/2-\delta/5}
        \end{align}
        by the definition of $I$. By \Cref{lem:taylor}, together with continuity of the action of the horospherical subgroup $U$, we get
        \[ d(u_{\reparam(m^{2-\delta} - r'_1,n^{2-\delta} - r'_2)}\omega', u_{P_{\br'}(m^{2-\delta} - r'_1,n^{2-\delta} - r'_2)}\omega') = O\left(\norm{(\log \bR)^{-1}}[\infty]\right) \]
        for all $\omega'\in U\omega$. Thus, \Cref{prop:equi-at-almost-squares} follows if we can show that for any $\omega'\in U\omega$ 
        \begin{align}\label{eq:first-appr} 
            \abs{\frac{1}{\#I}\sum_{(m,n)\in I}f(u_{P_{\br'}(m^{2-\delta} - r'_1,n^{2-\delta} - r'_2)}\omega ') - \int_{U\omega}f(z)\dd\mu_{\omega}(z)}[\bigg] \to 0, \quad \text{as } \norm{(\log \bR)^{-1}}[\infty] \to 0.
        \end{align}
        In order to simplify $P_{\br'}(m^{2-\delta} - r'_1,n^{2-\delta} - r'_2)$ for $(m,n)\in I$ we linearize the argument in $m$ and $n$ using Taylor approximation.
        
        \begin{lemma}\label{lem:linearize-argument}
            Let $R\geq 1$, $C>c>0$, $r'\in[cR,CR]$ and consider $J \defeq [r'^{1/(2-\delta)}, (r' + R^{1/2-\delta/5})^{1/(2-\delta)}]$. For any $r \in J$, we have
            \[ r^{2-\delta} - r' = l_{r'}(r) + O(R^{-2\delta/5}), \]
            where $l_{r'}(r) \defeq (2-\delta)r'^{\tfrac{1-\delta}{2-\delta}}(r - r'^{\tfrac{1}{2-\delta}})$ is linear in $r$. 
            Thus, for any $(m,n)\in I$, with $I$ as defined in \eqref{eq:integer-points}, we have
            \begin{align}\label{eq:bound-on-evaluation}
               \norm{P_{\br'}(m^{2-\delta} - r'_1,n^{2-\delta} - r_2')- P_{\br'}(l_{\br'}(m,n))}[\infty] &= O(\norm{\bR^{-2\delta/5}}[\infty])
            \end{align}
            where $l_{\br'}(m,n) \defeq (l_{r'_1}(m),l_{r_2'}(n))$.
        \end{lemma}
        \begin{proof}
            Consider the function $x\to x^{2-\delta}$ on $J$ and its Taylor expansion around $r'^{\tfrac{1}{2-\delta}}$ evaluated at~$r$, to obtain
            \[ r^{2-\delta} = r' + (2-\delta)r'^{\tfrac{1-\delta}{2-\delta}}(r - r'^{\tfrac{1}{2-\delta}}) + O\left(r'^{\tfrac{-\delta}{2-\delta}}(r-r'^{\tfrac{1}{2-\delta}})^2\right) = r' + l_{r'}(r) + O\left(r'^{\tfrac{-\delta}{2-\delta}}\abs{J}^2\right). \]
            Bounds for $\abs{J}$ may be obtained in the same way as the bounds in \eqref{eq:bigger-than-1}. These, together with the bounds on $r'$, yield $R^{-2\delta/5} \gg r'^{\tfrac{-\delta}{2-\delta}}\abs{J}^2\ \gg  R^{-2\delta/5}$. This establishes that
            \begin{align}\label{eq:linearization-bound}
                r^{2-\delta} - r' = l_{r'}(r) + O(R^{-2\delta/5})    
            \end{align}
            as claimed.
            
            The second statement follows by bounding the derivatives of $p_{t'_1,s'_1}$ and $p_{t'_2,s'_2}$ on $J_1$ and $J_2$ (defined in \eqref{eq:interval}), respectively. Since the argument is analogous for both, we just consider $p_{t'_1,s'_1}$ and $J_1$. By the mean value theorem and part one of the \namecref{lem:linearize-argument} we have for some $r\in[m^{2-\delta} -  r'_1 - 1, m^{2-\delta} -  r'_1 +1]$ that
            \[ \abs{p_{t'_1,s'_1}(m^{2-\delta} - r'_1) - p_{t'_1,s'_1}(l_{r'_1}(m))} = p_{t'_1,s'_1}'(r)(l_{r'_1}(m) - (m^{2-\delta} - r'_1)) = p_{t'_1,s'_1}'(r)O(R_1^{-2\delta/5}) \]
            if $R_1$ is large enough. Indeed, for sufficiently large $R_1$, we have $l_{r'_1}(m) \in [m^{2-\delta} - r'_1 - 1, m^{2-\delta} - r'_1 +1]$ by \eqref{eq:linearization-bound}. Since $m\in I_1$, it follows by \eqref{eq:int-close-to} that
            \[ [m^{2-\delta} - r'_1 - 1, m^{2-\delta} - r'_1 + 1] \subseteq [-(1 +R_1^{1/2-\delta/5}), 1 + R_1^{1/2-\delta/5}]\]
            Thus, it is enough to give an upper bound on the derivative $p_{t'_1,s'_1}'$ on the interval on the right hand side above. Recall that $p_{t'_1,s'_1}(r) = \sum_{\ell=0}^{d_1}c_{1,\ell}r^\ell$, where $c_{1,\ell} = (-1)^\ell{s'_1}^{\ell-1}{t'_1}^{-\ell}$. This implies that
            \[ p_{t'_1,s'_1}'(r) = \sum_{\ell=1}^{d_1}(-1)^\ell\ell{s'_1}^{\ell-1}{t'_1}^{-\ell}r^{\ell-1} \] 
            As $t'_1 = O(1 + 3(R_1)^{-A\delta})= O(1)$ and $s'_1r = O(R_1^{-1/2 + \delta/10}(1 + {R_1}^{1/2-\delta/5})) = O(1)$ it follows that $p_{t'_1,s'_1}'(r) = O(1)$ for all $r\in[-(1 +R_1^{1/2-\delta/5}), 1 + R_1^{1/2-\delta/5}]$.           
        \end{proof}

        By \Cref{lem:linearize-argument}, more precisely by \eqref{eq:bound-on-evaluation}, the statement in \eqref{eq:first-appr} follows, if we can show that
        \begin{align}\label{eq:second-appr} 
            \abs{\frac{1}{\#I}\sum_{(m,n)\in I}f(u_{P_{\br'}(l_{\br'}(m,n))}\omega ') - \int_{U\omega}f(z)\dd\mu_{\omega}(z)}[\bigg] \to 0, \quad \text{as } \norm{(\log\bR)^{-1}}[\infty] \to 0.
        \end{align}

        Finally, we reduce \eqref{eq:second-appr} to a statement about the torus $\TT^2$. Since $\omega = g\Gamma$ is a periodic $U$-orbit we know that $U\cap g\Gamma g^{-1}$ is a lattice in $U$, of covolume $L_\omega$. We may thus define the homomorphism 
        \[ h\colon U \to \TT^2 = \quot{\RR^2}{\ZZ^2}, \]
        given by the natural projection after identifying $\tquot{U}{U\cap g\Gamma g^{-1}}$ with $\TT^2$. More precisely, fixing a $\ZZ$-basis $\bm{b}_1=(b_{11},b_{12}), \bm{b}_2=(b_{21},b_{22})$ of $U\cap g\Gamma g^{-1}$ we define $h(u_{\bs}) = \bs B \in \TT^2$, where
        \[ B =\mat{b_{11}}{b_{12}}{b_{21}}{b_{22}}^{-1} = \frac{1}{L_\omega}\mat{b_{22}}{-b_{12}}{-b_{21}}{b_{11}} \]
        is the inverse of the matrix representation of the $\ZZ$-basis of $U\cap g\Gamma g^{-1}$. We have
        \begin{align}\label{eq:size-of-basis}
            L_\omega = \operatorname{covol}(U\cap g\Gamma g^{-1}) = \det(B^{-1}) \ll \norm{b_1}[\infty]\norm{b_2}[\infty].
        \end{align}
        Further, if $u\omega = u'\omega$, then $u$ and $u'$ differ by an element in $U\cap g\Gamma g^{-1}$ 
        and so $h(u) = h(u')$. Thus, for a fixed $\omega'\in U\omega$ we may define the map $\pi\colon U\omega'\to\TT^2$ by $\pi(u\omega') \defeq h(u)$, identifying the periodic orbit $U\omega'$ with $\TT^2$.

        Let $S_y(x) = y + x\in\TT^2$ be the translation map defined for any $y\in\RR^2$. Observe that we have
        \begin{align}\label{eq:equivariance}
            \pi(uu'\omega') = h(uu') = h(u) + h(u') = S_{h(u)}(\pi(u'\omega'))
        \end{align}
        for all $u,u'\in U$, since $h$ is a homomorphism. For a fixed function $f\in\Lip(X)$ as in \Cref{prop:equi-at-almost-squares} we define $\TTf\colon\TT^2\to\RR$ by $\TTf(x) \defeq f(\pi^{-1}(x))$. Applying Fourier decomposition to $\TTf$ we may write
        \[ \TTf(x) = \sum_{(m,n)\in\ZZ^2}a_{(m,n)}e_{(m,n)}(x), \]
        where $e_{(m,n)}(y) = \exp(2\pi i\langle (m,n),y\rangle)$ is an eigenfunction of the Laplacian on $\TT^2$ and
        \[ a_{(m,n)} = \int_{\TT^2}\TTf(y)e_{(m,n)}(-y)\dd y. \]
        In particular, we have
        \begin{align}\label{eq:0-fourier-coeff}
            a_{(0,0)} = \int_{\TT^2}\TTf(y)\dd y = \int_{U\omega}f(z)\dd\mu_{\omega}(z).
        \end{align}
        Let $\eps>0$. Since $f$ and $\omega$ are fixed, by standard Fourier analytic arguments there exist $K_1,K_2\in\NN$ such that for all $x\in\TT^2$
        \begin{align}\label{eq:TTf-approx}
            \abs{\TTf(x) - \sum_{\substack{\abs{k_1}\leq K_1\\\abs{k_2}\leq K_2}}a_{(k_1,k_2)}e_{(k_1,k_2)}(x)}[\bigg] \leq \eps.
        \end{align}
        Then, \eqref{eq:second-appr} follows immediately from the following \namecref{lem:weyl-sum-bound} in which we use the notations from above. Note that since $U\omega$ is a well-rounded periodic orbit, the lattice $B\ZZ^2$ is well-rounded.
        \begin{lemma}\label{lem:weyl-sum-bound}
            Let $\eps>0$ be given. For sufficiently large $\bR>1$ with $\norm{(\log\bR)^{-1}}[\infty] \to 0$ the following holds. For every $(k_1,k_2)\neq(0,0)\in\ZZ^2$ with $\abs{k_1}\leq K_1$ and $\abs{k_2}\leq K_2$, we have
            \[ \abs{\frac{1}{\#I}\sum_{(m,n)\in I}e_{(k_1,k_2)}\big(P(l_{\br'}(m,n))B\big)}[\bigg] < \frac{\eps}{4K_1K_2}\]
            where
            \[ l_{\br'}(m,n)= \left((2-\delta)(r'_1)^{\tfrac{1-\delta}{2-\delta}}(m - (r'_1)^{\tfrac{1}{2-\delta}}),(2-\delta)(r'_2)^{\tfrac{1-\delta}{2-\delta}}(n - (r'_2)^{\tfrac{1}{2-\delta}})\right), \]
            and $B\ZZ^2$ is a well-rounded lattice.
        \end{lemma}
        Before we prove the \namecref{lem:weyl-sum-bound} we finish the proof of \Cref{prop:equi-at-almost-squares}. First, note that using \eqref{eq:equivariance} it holds that
        \begin{align}\label{eq:transferrence}
            \sum_{(m,n)\in I} f(u_{P(l_{\br'}(m,n))}\omega') = \sum_{(m,n)\in I} \TTf(S_{h(u_{P(l_{\br'}(m,n))})}(0,0)).
        \end{align} 
        Further, for $(m,n)\in I$ we have $e_{(k_1,k_2)}(S_{h(u_{P(l_{\br'}(m,n))})}(0,0)) = e_{(k_1,k_2)}\big(P_{\br'}(l_{\br'}(m,n))B\big)$. Thus, by \eqref{eq:0-fourier-coeff}, \eqref{eq:TTf-approx} and \eqref{eq:transferrence}, exchanging sums, applying the triangle inequality and imposing \Cref{lem:weyl-sum-bound} we obtain
        \[ \abs{\frac{1}{\#I}\sum_{(m,n)\in I}f(u_{P_{\br'}(l_{\br'}(m,n))}\omega ') - \int_{U\omega}f(z)\dd\mu_{\omega}(z)}[\bigg] \leq 2\eps. \]
        If $\norm{\bR^{-1}}[\infty] \to 0$, the parameter $\eps$ can be chosen arbitrarily close to $0$, proving \eqref{eq:second-appr}.\\
        This finishes the proof of \Cref{prop:equi-at-almost-squares}.
    \end{proof}

    \begin{proof}[Proof of \Cref{lem:weyl-sum-bound}]
        Recall that $P_{\br'}(r_1,r_2) = (p_{t'_1,s'_1}(r_1),p_{t'_2,s'_2}(r_2)) = \left(\sum_{\ell=0}^{d_1}c_{1,\ell}r_1^\ell,\sum_{\ell=0}^{d_2}c_{2,\ell} r_2^\ell\right)$ and
        \[ l_{\br'}(m,n) = \left(l_{r'_1}(m),l_{r'_2}(n)\right) = \left((2-\delta){r'_1}^{\tfrac{1-\delta}{2-\delta}}(m - {r'_1}^{\tfrac{1}{2-\delta}}),(2-\delta){r'_2}^{\tfrac{1-\delta}{2-\delta}}(n - {r'_2}^{\tfrac{1}{2-\delta}})\right). \]
        Hence, we can write
        \begin{align*} 
            e\big(P_{\br'}(l_{\br'}(m,n))B(k_1,k_2)^t\big) &= e\big(p_{t'_1,s'_1}(l_{r'_1}(m))b_1'\big)\cdot e\big(p_{t'_2,s'_2}(l_{r'_2}(n))b_2'\big),
        \end{align*}
        where $b_1',b_2'\in\RR$ are defined by $B(k_1,k_2)^t = (b_1', b_2')^t$.
        Clearly $\norm{(b_1',b_2')^t}[\infty]\leq \norm{(b_1',b_2')^t}[2] \leq \sqrt{2}\norm{(b_1',b_2')^t}[\infty]$, and since any non-zero element in $B\ZZ^2$ has length at least $\lambda_1(B\ZZ^2)$ we have by the assumption on well-roundedness of $B\ZZ^2$ and by \eqref{eq:succ-bounds} that
        \begin{align}\label{eq:bounds-on-b}
            \sqrt{\frac{2}{3}} L_\omega^{-1/2} \leq \max\{\abs{b_1'}, \abs{b_2'}\} \leq \sqrt{2}\max\{K_1,K_2\}L_\omega^{-1/2}.
        \end{align}
        We further use the product structure $I = I_1\times I_2$, to get
        \begin{align*}
            \abs{\frac{1}{\#I}\sum_{(n,m)\in I}e\big(P(l_{\br'}(m,n))B(k_1,k_2)^t\big)}[\bigg]&\\
            \leq \abs{\frac{1}{\#I_1}\sum_{m\in I_1}e\big(p_{t'_1,s'_1}(l_{r'_1}(m))b_1')}[\bigg]&\abs{\frac{1}{\#I_2}\sum_{n\in I_2}e\big(p_{t'_2,s'_2}(l_{r'_2}(n))b_2'\big)}[\bigg].
        \end{align*}
        The trivial bound is
        \begin{align*}
            \abs{\frac{1}{\#I_i}\sum_{m\in I_i}e\big(p_{t'_i,s'_i}(l_{r'_i}(m))b_i'\big)}[\bigg] \leq 1,
        \end{align*}
        for $i=1,2$. For each $(k_1,k_2)\neq 0$ we choose $i=1,2$ so that $\norm{b_i'}[\infty] = \max\{\abs{b_1'}, \abs{b_2'}\}$. Then, it is sufficient to show that
        \begin{align}\label{eq:weyl-bound}
            \abs{\frac{1}{\#I_i}\sum_{m\in I_i}e\big(p_{t'_i,s'_i}(l_{r'_i}(m))b_i'\big)}[\bigg] < \frac{\eps}{4K_1K_2}.
        \end{align}
        Since $p_{t'_i,s'_i}(m) = \sum_{\ell=0}^{d}c_{\ell,1}m^\ell$, we can write
        \[ p_{t'_i,s'_i}(l_{r'_i}(m))b_i' = \sum_{\ell=0}^{d}b_i'c_{\ell,1}(2-\delta)^\ell {r'_i}^{\tfrac{(1-\delta)\ell}{2-\delta}}(k-{r'_i}^{\tfrac{1}{2-\delta}})^\ell = \widetilde{p}(k-{r'_i}^{\tfrac{1}{2-\delta}}), \]
        where $\widetilde{p}(m) = b_i'\sum_{\ell=0}^{d}\widetilde{c}_{\ell}m^\ell$ with $\widetilde{c}_\ell = c_{\ell,1}(2-\delta)^\ell {r'_i}^{\tfrac{(1-\delta)\ell}{2-\delta}}$. We will finish the proof using the following quantitative version of Weyl's inequality.

        The following reformulation of Weyl's inequality is given in \cite[Proposition 4.3, Lemma 4.4]{GT12}.
        
        \begin{lemma}[Weyl's inequality]\label{lem:weyls-inequality}
            Let $p(m) = \sum_{\ell=0}^{d} c'_\ell m^\ell$ be a polynomial of degree $d$. If
            \[ \abs{\sum_{k\in I'} e\big(p(k-a)\big)}[\bigg] \geq \rho(\# I'), \]
            for an integer interval $I' = \NN\cap[a,b]$ and some $\rho\in(0,\tfrac{1}{2})$, then there exists an integer $1\leq q\leq \rho^{-O_d(1)}$ such that
            \[ \norm{qc'_\ell}[\ZZ] \leq \rho^{-O_d(1)}(\# I')^{-\ell}, \]
            for all $1\leq \ell \leq d$. Here, $\norm{~\cdot~}[\ZZ]$ denotes the distance to the nearest integer.
        \end{lemma}

        We want to apply \Cref{lem:weyls-inequality} with $p = \widetilde{p}$, $I' = I_i$ and $\rho = \tfrac{\eps}{4K_1K_2}$. In order to deduce \eqref{eq:weyl-bound}, it suffices to show that for some $1\leq\ell_0\leq d$ and all $1\leq q \leq \rho^{-O_d(1)}$ we have
        \begin{align}\label{eq:bound-on-coeff}
            \norm{qb_i'\widetilde{c}_{\ell_0}}[\ZZ] > \rho^{-O_d(1)}(\# I)^{-\ell_0}.
        \end{align}

        First, let us give a bound on $\abs{\widetilde{c}_{\ell}}$. Using that $c_{\ell,1} = (-1)^\ell {s'_i}^{\ell-1}{t'_i}^{-\ell}$ and that $t'_i = O(1)$ we can find a constant $C'=C'(\delta,d,c,C)$ such that
        \begin{align}\label{eq:bounds-on-coeff}
            {C'}^{-1}({t'_i}^{-1}s'_i)^{\ell-1}{r'_i}^{\tfrac{(1-\delta)\ell}{2-\delta}} \leq \abs{\widetilde{c}_{\ell}} \leq C'({t'_i}^{-1}s'_i)^{\ell-1}{r'_i}^{\tfrac{(1-\delta)\ell}{2-\delta}}.
        \end{align}
        We further know that ${t'_i}^{-1}s'_i = O(R_i^{-1/2+\delta/10})$ and $r'_i\in[cR_i,CR_i]$. Combining this with the bounds in \eqref{eq:bounds-on-coeff} and the lower bound in \eqref{eq:bigger-than-1} we get, for every $2\leq \ell \leq d$, that
        \begin{align}\label{eq:lower-bound}
            \begin{split}
                \abs{\widetilde{c}_{\ell}}(\# I)^{\ell} &\gg ({t'_i}^{-1}s'_i)^{\ell-1}R_i^{\tfrac{(1-\delta)\ell}{2-\delta}}(\# I)^{\ell}\\
                &\gg  R_i^{(-1/2+\delta/10)(\ell-1)}R_i^{\tfrac{(1-\delta)\ell}{2-\delta}}R_i^{\tfrac{(\delta+2\delta^2)\ell}{10(2-\delta)}}\\
                &\gg R_i^{\tfrac{10-7\delta+\delta^2 -\delta(2-\delta)\ell}{10(2-\delta)}}
            \end{split}
        \end{align}
        Analogously, using the upper bounds in \eqref{eq:bounds-on-coeff} and \eqref{eq:bigger-than-1}, we obtain
        \begin{align}\label{eq:upper-bound}
            \abs{\widetilde{c}_{\ell}}(\#I)^{\ell} &\ll R^{\tfrac{10-7\delta+\delta^2 -\delta(2-\delta)\ell}{10(2-\delta)}}.
        \end{align}
        We define $\ell_0 \defeq \lfloor{\tfrac{10-7\delta+\delta^2}{\delta(2-\delta)}}\rfloor\geq2$ (the inequality follows from a straightforward calculation) if $\tfrac{10-7\delta+\delta^2}{\delta(2-\delta)}\not\in\NN$ and $\ell_0 \defeq \tfrac{10-7\delta+\delta^2}{\delta(2-\delta)} - 1$ else. In both cases, we obtain
        \[ \delta(2-\delta) \geq 10-7\delta+\delta^2-\delta(2-\delta)\ell_0 > 0. \]
        Thus, we can set $\eta \defeq\tfrac{10-7\delta+\delta^2 -\delta(2-\delta)\ell_0}{10(2-\delta)}\in(0,\tfrac{\delta}{10}]$ and use \eqref{eq:lower-bound} and \eqref{eq:upper-bound} to get that
        \[ {C''}^{-1}R_i^{\eta} \leq \abs{\widetilde{c}_{\ell_0}}(\#I)^{\ell_0} \leq {C''}R_i^{\eta} \]
        for some $C'' = C''(\delta,d,C,c,C')$. The lower bounds on $\ell_0$ and on $\#I$ given in \eqref{eq:bigger-than-1} and the upper bound on $\eta$ then imply that
        \[ \abs{\widetilde{c}_{\ell_0}} \ll R_i^{\eta-\tfrac{2(\delta+2\delta^2)}{10(2-\delta)}} \ll R_i^{\tfrac{-\delta^2}{2(2-\delta)}}. \]
        
        If $R_i$ is sufficiently large, so that $\frac{C''^{-1}\sqrt{2}}{\sqrt{3}}R_i^{\eta/2} > (\tfrac{\eps}{4K_1K_2})^{-O_d(1)} = \rho^{-O_d(1)}$ and $R_i^{\eta/2}\geq L_{\omega}$, we have for all $1\leq q \leq \log R_i$ that
        \[ \norm{qb_i'\widetilde{c}_{\ell_0}}[\ZZ] \geq C''^{-1}R^{\eta/2}(\# I)^{-\ell_0} > \rho^{-O_d(1)}(\# I)^{-\ell_0}, \]
        proving \eqref{eq:bound-on-coeff}. Indeed, we can first use the upper bounds on $\abs{\widetilde{c}_{\ell_0}}$, $q$, and $b_i'$, to conclude that 
        \[ \abs{qb_i'\widetilde{c}_{\ell_0}} \leq q\cdot 2\max\{K_1,K_2\}L_\omega^{-1/2}R_i^{\tfrac{-\delta^2}{2(2-\delta)}} \leq 1/2 \]
        $1\leq q\leq \log R_i$, for sufficiently large $R_i$. If $R_i$ is large enough so that $R_i^{\eta}\geq L_\omega$, we can use the lower bounds on $\widetilde{c}_{\ell_0}(\# I)^{\ell_0}$, $q$ and $b_i'\geq \sqrt{\frac{2}{3}}L_{\omega}^{-1/2}\geq R^{-\eta/2}$ to obtain
        \[ \norm{qb_i'\widetilde{c}_{\ell_0}}[\ZZ] = \abs{qb_i'\widetilde{c}_{\ell_0}} \geq \frac{C''^{-1}\sqrt{2}}{\sqrt{3}}R^{\eta/2}(\#I)^{-\ell_0}.  \]
        This finishes the proof of \Cref{lem:weyl-sum-bound} and thus of \Cref{prop:equi-at-almost-squares}.
    \end{proof}

    \begin{remark}
        Recall that the parameter $\bs'$ defined in \eqref{eq:def-x'} and provided by \eqref{prop:approaching-omega} essentially determines the distances between the $U$-orbits of $u\_\br^kx$ and $\omega$. In particular, this parameter essentially determines the range $\rect{(\bR^k)^{1/2-\delta/5}}$ for which we can shadow the $U$-orbit $\omega$ depicted in \Cref{cor:staying-close-to-closed-orbits} (in particular, see \eqref{eq:bound-on-v_s}). Thus, one might expect very small parameters $\bs'$ to be beneficial for shadowing the periodic $U$-orbit $U\omega$. However, in order to apply Weyl's inequality (\Cref{lem:weyls-inequality}) we need to obtain the lower bound on $\abs{\widetilde{c}_{\ell}}(\# I)^{\ell}$ in \eqref{eq:lower-bound}. To obtain this lower bound the parameter $\bs'$ may not be too small.         
        Similarly, the assumption on well-roundedness of the perioidic orbit $U\omega$ allows to obtain the bounds on $b_i'$ in \eqref{eq:bounds-on-b}. These, too, are necessary in order to apply Weyl's inequality.
    \end{remark}

	\addcontentsline{toc}{section}{References}
	\printbibliography[title=References]
\end{document}